\documentclass{amsproc}

\usepackage{amssymb}
\usepackage{tikz}

\usepackage{graphicx}
\usepackage{amsmath}
\usepackage{subfigure}
\usepackage{calc}
\usepackage{float}

\newtheorem{theorem}{Theorem}[section]
\newtheorem{lemma}{Lemma}[section]
\newtheorem{proposition}{Proposition}[section]
\newtheorem{remark}{Remark}[section]
\newtheorem{corollary}{Corollary}[section]
\newtheorem{definition}{Definition}[section]

\def\Z{\mathbb{Z}}
\def\R{\mathbb{R}}
\def\C{\mathbb{C}}

\def\N{\mathbb{N}}
\def\F{\mathcal{F}}
\def\L{\mathrm{L}}

\def\0{{\bf{0}}}

\def\vx{{\vec{\xi}}}
\def\vt{{\vec{t}}}
\def\vs{{\vec{s}}}

\def\-{{\mbox{\tiny $ - $ }}}
\def\+{{\mbox{\tiny $ + $ }}}

\def\0{{\mbox{\tiny $ (0) $ }}}
\def\1{{\mbox{\tiny $ (1) $ }}}
\def\2{{\mbox{\tiny $ (2) $ }}}

\newcommand{\dotcup}{\ensuremath{\mathaccent\cdot\cup}}
\newcommand{\bigdotcup}{\bigcup \! \! \! \! \!  \cdot \ \ }

\numberwithin{equation}{section}

\usepackage{lipsum}

\newcommand\blfootnote[1]{
  \begingroup
  \renewcommand\thefootnote{}\footnote{#1}%
  \addtocounter{footnote}{-1}%
  \endgroup
}

\usepackage{pstricks}
\definecolor{gray7}{rgb}{0.7,0.7,0.7}
\definecolor{gray8}{rgb}{0.8,0.8,0.8}
\definecolor{gray9}{rgb}{0.9,0.9,0.9}

\begin{document}

\allowdisplaybreaks

   \begin{center}
      \Large\textbf{Equations For Parseval's Frame Wavelets In $L^2(\R^d)$\\  With Compact Supports}
   \end{center}

\vspace{.5cm}

\begin{center}
      \large\textit{Xingde Dai}
\end{center}

\bigskip

\begin{center}Abstract\end{center}

\begin{quotation}

{\small Let $d\geq 1$ be a natural number and $A_0$ be a $d\times d$ expansive integral matrix with determinant $\pm 2.$
Then $A_0$ is integrally similar to an integral matrix $A$ with certain additional properties.
A finite solution to the system of equations  associated with the matrix $A$ will result in an iterated sequence $\{\Psi^k \chi_{[0,1)^d}\}$ that converges to a function $\varphi_A$ in $L^2(\R^d)$-norm. With this (scaling) function $\varphi_A,$ we will construct the Parseval's wavelet function $\psi$ with compact support associated with matrix $A_0.$
}

\end{quotation}

      \address{The University of North Carolina at Charlotte}
      \email{xdai@uncc.edu}



\bigskip

\section{Introduction}

\bigskip

In this paper, $d$ is a natural and $d\geq 1$, $\R$ is the real numbers and $\C$ is the complex numbers.  $\R^d$ will be the $d$-dimensional  real Euclidean space and
and $\C^d$  will be the $d$-dimensional complex  Euclidean space.
$\R^d$ is a subset of $\C^d.$ Let $\{\vec{e}_j,j=1,\cdots,d\}$ be the standard basis.
For two vectors $\vec{x}=\sum_{j=1} ^d x_j \vec{e}_j$  and $\vec{y}=\sum_{j=1} ^d y_j \vec{e}_j,$ the inner product of $\vec{x}$ and $\vec{y}$ is  $\vec{x} \circ \vec{y} \equiv \sum_{j=1} ^d x_j \bar{y}_j.$
For a vector $\vx\in\C^d,$
its real part $\mathfrak{Re}(\vx)$ and imaginary part $\mathfrak{Im}(\vx)$ are
vectors in $\R^d$  with $\vx = \mathfrak{Re} (\vx) + i \mathfrak{Im} (\vx).$
The measure $\mu$ will be the Lebesgue measure on $\R^d$
and $L^2(\R^d)$ will be the Hilbert space of all square integrable functions on $\R^d$.
A countable set of elements $\{\psi_i: i\in \Lambda\}$ in $L^2(\R^d)$ is called a \emph{normalized tight frame} of $L^2(\R^d)$ if
\begin{equation}
\label{eq:frame}
   \sum_{i\in\Lambda} |\langle f, \psi_i \rangle |^2 = \|f\|^2,
\ \forall f\in L^2(\R^d).
\end{equation}
It is well known in the literature \cite{christensen} that the Equation \eqref{eq:frame} is equivalent to
\begin{equation}
\label{eq:frameeq}
   f = \sum_{i\in\Lambda} \langle f, \psi_i \rangle \psi_i ,
\ \forall f\in L^2(\R^d),
\end{equation}
where the convergence is in $L^2(\R^d)$-norm.  For a  vector $\vec{\ell}\in\R^d,$
the \textit{translation operator} $T_{\vec{\ell}}$ is defined as
\begin{eqnarray*}
    (T_{\vec{\ell}} f)(\vec{t}) &\equiv& f(\vec{t}-\vec{\ell}), ~ \forall f\in L^2(\R^d).\\
\end{eqnarray*}
Let $A$ be a $d \times d$ integral matrix with eighenvalues $\{\beta_1,\cdots \beta_d\}$. $A$ is called \textit{expansive} if $\min \{|\beta_j|, j\leq d\} >1$. The norm of the linear transformation $A$ on $\R^d$ (or $\C^d$) will be $\|A\|\equiv \max\{|\beta_j|,j=1,\cdots,d\}.$
For two
vectors $\vec{t}_1,\vec{t}_2$ in the Euclidean space $\C^d,$  we have $\vec{t}_1\circ A \vec{t}_2 = A^\tau \vec{t}_1 \circ \vec{t}_2,$ where $A^\tau$ is the transpose matrix of the \textit{real} matrix $A$. $A^\tau$ is expansive iff $A$ is.
We define operator $U_{A}$ as
\begin{eqnarray*}
    (U_A f)(\vec{t}) &\equiv& (\sqrt{|\det(A)|}) f(A\vec{t}),~\forall f \in L^2(\R^d).
\end{eqnarray*}
This is a unitary operator. In particular, for an expansive matrix $A$ with $|\det(A)|=2,$ we will use
$D_A$ for $ U_A$ and call it the \textit{dilation operator associated with $A$}.
The integral lattice $\Z^d$ is an Abelian group under vector addition. The subset $(2\Z)^d$ is a subgroup. For a fixed $d\times d$ integral matrix $A$ with $|\det(A)| =2,$
the two sets $A \Z^d$ and $A^\tau \Z^d$ are proper subgroups of $\Z^d$ containing $(2\Z)^d$. An integral matrix $B$ is said \textit{integrally similar} to another integral matrix $C$ by matrix $S$ if
$C=SB S^{-1} $ for an integral matrix $S$ with $S^{-1}$ being also integral.

\begin{definition}
\label{def:defpsi0}
Let $A$ be an expansive integral matrix  with $|\det(A)|=2.$
A function $\psi\in L^2(\R^d)$ is called a Parseval's frame wavelet associated with $A$, if the set
\begin{equation}
    \{ D_A ^n T_{\vec{\ell}}\psi,  n\in\Z,\vec{\ell}\in\Z^d\}
\end{equation}
constitutes a normalized tight frame of $L^2(\R^d).$
\end{definition}
\begin{remark}
The function $\psi$ is also called single function normalized tight frame wavelet.
A Parseval's frame wavelet is not necessarily a unit vector in $L^2(\R^d)$ unless it is an orthonormal  wavelet. By definition an element $\psi\in L^2(\R^d)$ is
a Parseval's frame wavelet if and only if
\begin{equation}
\label{eq:framew}
    \|f\|^2 = \sum_{n\in\Z , \vec{\ell} \in \Z^d} |\langle f,D_A ^n T_{\vec{\ell}}\psi \rangle |^2, \ \forall f\in L^2(\R^d).
\end{equation}
By Equation \eqref{eq:frameeq} this is equivalent to
\begin{equation}
\label{eq:frameweq}
    f = \sum_{n\in\Z , \vec{\ell} \in \Z^d} \langle f,D_A ^n T_{\vec{\ell}}\psi \rangle D_A ^n T_{\vec{\ell}}\psi , \   \forall f\in L^2(\R^d).
\end{equation}
\end{remark}

In this paper, we will prove that an expansive integral matrix $A_0$ with determinant value $\pm 2$ is integrally similar to an integral matrix $A$ with certain additional properties.
We will prove that a finite solution to the system of equations \eqref{eq:lawtoneq} associated with the matrix $A$ will result in an iterated sequence $\{\Psi^k \chi_{[0,1)^d}\}$ that converges to a function $\varphi_A.$ With this  function $\varphi_A,$ we will construct the Parseval's wavelet function $\psi$ with compact support associated with matrix $A_0.$

The original ideas of the above construction were due to W. Lawton on the one dimensional model \cite{lawton} in 1990.
People use the ideas, however, there is no paper clearly state and prove the above constructions in general $d$ dimensional case, the generalization of Lawtton's results into $L^2(\R^d).$
This paper give a clear formulation together with a rigorous proof.

In \cite{dai2d} we state and prove the above construction in the case $L^2(\R^2)$, the two dimensional case.
It is natural that in some reasoning in this paper which is for general cases $d\geq 1$
we use the similar ideas for the corresponding conclusions in a previous paper \cite{dai2d}.
In order to control the size of this paper and also to make this paper readable, we placed the proofs to some eleven statements into the Appendix. The 10 page Appendix includes proof of
Theorem \ref{theom:redsym}; proofs of Propositions  \ref{prop:filter}, \ref{prop:compactsupport}, \ref{prop:telleskope}, \ref{prop:go0} and \ref{prop:isnorm};
proofs of Lemmas \ref{lem:LV}, \ref{lem:see0}, \ref{lem:gconverge}, \ref{lem:ineq} and \ref{lem:identitylj}.
We have to note that the above mentioned proofs are different and improved somehow than corresponding parts in
\cite{dai2d}.

In the one dimensional case, the $1 \times 1$ dilation matrix $2$ is simple: Its transpose is itself; it acts on the integer lattice $\Z$ will result in an simple sublattice $2\Z$, the even integers; and the integer lattice $\Z$ has a partition of even and odd numbers. However, for a $d \times d$ integral matrix $A$, these things have been changed. For example, $A\Z^d$ and $A^\tau \Z^d$ can be different sublattices of $\Z^d.$ The desired property $A\Z^d=A^\tau \Z^d$ will make the situation relatively simpler when our discussions are in  the frequency domain and the time domain at same time. We formulated the properties of matrices and related lattices we will need in higher dimensional cases into the Partition
Theorem (Theorem \ref{theom:that}).
These properties have been cited in some key Propositions and their proofs (Propositions \ref{prop:filter}, \ref{prop:l2phi}, \ref{prop:telleskope}).
In particular, we need these properties when we define the single function Parseval's frame wavelet (Definition \ref{def:psi}).
Sections \ref{ss:Matrices} and \ref{ss:s2} are devoted to prepare and prove this Partition Theorem.

The next fundamentally important issue is to show that the scaling function $\varphi_A$ is in $L^2(\R^d)$. Due to the variety of the matrices we have to cover, we insert technical Lemmas (Lemmas \ref{lem:sesami}, \ref{lem:A}, \ref{lem:C}) and Proposition \ref{prop:l2phi} in Section \ref{ss:phil2}.
We will follow the classical method for constructing such frame wavelets as provided by I. Daubechies in \cite{dau}. That is, from the filter function $m_0$ to the scaling function $\varphi$. This is outlined in Section \ref{ss:mainresults} Steps (C$_1$)-(C$_3$).  Then we construct the
wavelet function $\psi$. To construct the filter function $m_0$ we start with the system of equations \eqref{eq:lawtoneq} which is a direct generalization of
W. Lawton's system of equations \cite{lawton} for frame wavelets in $L^2(\R)$.
In Section \ref{ss:examples}, we provide two examples to show that the constructions we proved in this paper does produce Parseval's frame wavelets, even orthonormal wavelets.

The literature  of wavelet theory in higher dimensions is rich. Many authors provide significant contributions to the theory.
It is hard to make a short list. However, the author must cite the following names and their papers.

Q. Gu and D. Han \cite{guhan} proved that, if an integral expansive matrix is associated with single function \textit{orthogonal wavelets} with \textit{multi-resolution analysis} (MRA), then the matrix determinant must be $\pm 2.$ Orthogonal wavelets are special single function Parseval's frame wavelets.

The existence of Haar type orthonormal   wavelets (hence with compact support) in $L^2(\R^2)$ was proved by J. Lagarias and
Y. Wang in \cite{wangyang2}.
The first examples of such functions with compact support and with properties of
high smoothness in  $L^2(\R^2)$ and $L^2(\R^3)$ were provided by E. Belogay and Y. Wang
in  \cite{wangyang}.
Compare with \cite{wangyang2} and \cite{wangyang} our methods appear to be more constructive. Also, it provides variety for single function Parseval wavelets which includes the orthogonal wavelets. The methods in this paper provide a wide base in searching for more frame wavelets with normal properties as wavelets in \cite{wangyang}.

The scaling function $\varphi$ in this paper is not necessarily orthogonal. So the wavelet system constructed fits the definition of the
\textit{frame multi-resolution analysis} (FMRA) by   J. Benedetto and S. Li in \cite{benedetto} and  it also fits the definition of the \textit{general multi-resolution analysis} (GMRA) by
L. Baggett, H. Medina and K.  Merrill in \cite{baggett}.

\bigskip

\section{Main Results}\label{ss:mainresults}

Let $A_0$ be a $d\times d$ expansive integral matrix with $|\det (A_0)|=2.$ We will construct Parseval's frame wavelets associated with $A_0$ in the following steps (A)-(E).\\

(A) Find a $d\times d$ integral matrix $A$ (Partition Theorem) which is integrally similar to $A_0$  with the following properties,
\begin{enumerate}
  \item
  \begin{equation*}
  S ^{-1} A S = A_0,
  \end{equation*}
  where $S$ is an integral matrix with $|\det (S)|=1.$
  \item
    \begin{equation*}
    A \Z^d = A^\tau \Z^d.
    \end{equation*}
  \item There exists a vector $\vec{\ell}_A\in\Z^d$ such that
\begin{equation*}
 \Z^d = (\vec{\ell}_A + A \Z^d) \dotcup A \Z^d.
\end{equation*}
  \item There exists a vector $\vec{q}_A \in \Z^d$\\
\begin{equation*}
  \vec{q}_A \circ A \Z^d \subseteq 2\Z
 \textit{ { and \ }}
        \vec{q}_A \circ (\vec{\ell}_A + A \Z^d) \subseteq 2\Z+1.
\end{equation*}
\end{enumerate}

\bigskip

(B) Solve the system of equations
\begin{equation*}
\left\{\begin{array}{l}
\sum_{\vec{n}\in\Z^d}h_{\vec{n}}\overline{h_{\vec{n}+\vec{k}}}=\delta_{\vec{0} \vec{k}},~ \vec{k}\in A\Z^d \\
\sum_{\vec{n}\in\Z^d}h_{\vec{n}}=\sqrt{2}.
\end{array}\right.
\end{equation*}
for a finite solution $\mathcal{S}=\{h_{\vec{n}}:~{\vec{n}\in\Z^d}\}.$ We say $\mathcal{S}$ is a finite solution if the index set of non-zero terms $h_{\vec{n}}$ is included in
the set $\Lambda_0 \equiv \Z^d \cap [-N_0,N_0]^d$ for some natural number $N_0.$

\bigskip

(C) Let $\Psi$ be the linear operator on $L^2(\R^d)$ by
\begin{equation*}
  \Psi \equiv \sum_{\vec{n}\in\Lambda_0} h_{\vec{n}} D_AT_{\vec{n}}.
\end{equation*}
The iterated sequence $\{\Psi^k \chi_{[0,1)^d}, k\in\N\}$ will converge to the scaling function $\varphi_A$ in the $L^2(\R^d)$-norm (Theorem \ref{theom:iterate}).

\bigskip

(D)
Define function  $\psi_A$ 
\begin{eqnarray*}
  \psi_A &\equiv& \sum_{\vec{n}\in\Z^d}
  (-1)^{\vec{q}_A \circ \vec{n}} \overline{h_{\vec{\ell}_A-\vec{n}}} D_AT_{\vec{n}} \varphi_A .
\end{eqnarray*}
This is a Parseval's frame wavelet with compact support associated with matrix $A$ (Theorem \ref{theom_frame}).

\bigskip

(E) Define the wavelet function $\psi$ by
\begin{equation*}
    \psi(\vec{t}) \equiv \psi_A(S\vec{t}),\forall \vec{t}\in\R^d.
\end{equation*}
The function $\psi$ is a Parseval's frame wavelet with compact support associated with the given matrix $A_0$ (Theorem \ref{theom:redsym}).

\bigskip

To prove the conclusion in step (C) and that $\varphi_A$ has a compact support,  we do the following steps (C$_1$)-(C$_3$).\\

(C$_1$) Define the filter function $m_0(\vt)$.
\begin{equation*}
    m_0(\vec{t})\equiv \frac{1}{\sqrt{2}}\sum_{\vec{n}\in\Lambda_0} h_{\vec{n}} e^{-i\vec{n}\circ\vec{t}}.
 \end{equation*}

\bigskip

(C$_2$) Define a function $g$ by
\begin{equation*}
 g(\vec{\xi})\equiv \frac{1}{(2\pi)^{d/2}}\prod_{j=1}^{\infty}m_0((A^{\tau})^{-j} \vec{\xi}), \forall \vec{\xi}\in\R^d
\end{equation*}
The function $g$ is an $L^2(\R^d)$-function (Proposition \ref{prop:l2phi}), and its extension is an entire function on $\C^d$  (Proposition \ref{prop:entire}).

\bigskip

(C$_3$) Define the scaling function $\varphi_A$ 
\begin{equation*}
\varphi_A \equiv \F^{-1} g,
\end{equation*}
where $\F$ is the Fourier transform.
The scaling function $\varphi_A$ is an $L^2(\R^d)$-function with compact support (Proposition \ref{prop:compactsupport}).

\bigskip

\section{Factorization of Matrices}\label{ss:Matrices}

In this section, we will prove a factorization theorem (Proposition \ref{prop:see02}) of integral matrices we discuss.
This will provide a base to prove the Partition Theorem in Section \ref{ss:s2}.
 Let $\mathfrak{M}$ be the set of $d \times d$ integral matrices with determinants $\pm 1.$
This is a group under matrix multiplication.
We will use some elementary  $d\times d$ integral  matrices. We state their properties in  Lemma~\ref{lem:prmat}.

Let $p,i,j$ be natural numbers and $ 1\leq p,i,j\leq d$.
Let $I$  be the $d\times d$ identity matrix.
Denote $\Delta_{i j}$  the $d\times d$ matrix with $1$ at $i j$ position and $0$ at all other positions.
Denote $I_{i j}$ be the matrix after interchanging the $i^{\texttt{th}}$ and the $j^{\texttt{th}}$ columns in the identity matrix $I$.
Denote $D_p \equiv I+\Delta_{p p},$  and $S_p \equiv I-2\Delta_{p p}.$

 \begin{lemma}
 \label{lem:prmat}
 Let $B$ be a $d\times d$ matrix, and $m,n,p,k,\ell,i,j$ be natural numbers in $[1,d].$
\begin{enumerate}  \item
  \begin{equation}
  \label{eq:magic}
\Delta_{m n} \cdot \Delta_{k\ell} = \delta_{n k} \Delta_{m \ell}
 \end{equation}
where $\delta$ is the Kronecker delta.
 \item If $i\neq j, $
 then
\begin{equation}
\label{eq:f2}
 (I+\Delta_{i j})^{-1} = I-\Delta_{i j}.
\end{equation}

    Right multiply $(I\pm \Delta_{i j})$ to $B$ will add the $i^{\texttt{th}}$ column to the $j^{\texttt{th}}$ in matrix $B$ (subtract the $i^{\texttt{th}}$ column from the $j^{\texttt{th}}$ column in matrix $B$).
\item $S_p =S_p ^{-1}.$
   Right multiply $S_p$ to $B$ will change the sign of the $p^{\texttt{th}}$ column.
  \item $I_{i j}   = I_{j i} = I_{i j} ^{-1}.$
   Right multiply $I_{i j}$ to $B$ will interchange the $i^{\texttt{th}}$ and the $j^{\texttt{th}}$ columns in matrix $B$.
  \item $D_p ^{-1} = I - \frac{1}{2} \Delta_{p p} \notin \mathfrak{M}.$
   Right multiply $D_p$ to $B$ will multiply the $p^{\texttt{th}}$ column in $B$ by $2$.
\end{enumerate}
\end{lemma}

An integral matrix $B$ is said \textit{integrally similar} to another integral matrix $C$ by matrix $S$ if
$C=SB S^{-1} $ for an integral matrix $S$ with $S^{-1}$ being also integral. The only possible integral matrices to serve for
$S$ are matrices in the group $\mathfrak{M}.$ We will call the matrices in the set
\begin{equation*}
\mathfrak{G} = \{S_p\}\cup
\{(I\pm \Delta_{i j}), i\neq j\} \cup
\{I_{i j}  \}
\end{equation*}
\textit{elementary matrices.} Every member in $\mathfrak{G}$ has its inverse also in $\mathfrak{G}$.
Also, we have $\mathfrak{G} \subset \mathfrak{M}.$\\

\begin{lemma}
\label{lem:LV}
Let $B$ be a $d \times d$ non singular integral matrix. Then $B$ can be factored into
\begin{equation*}
    B=LV
\end{equation*}
where $V\in \mathfrak{M}$ which is the finite product of elements from $\mathfrak{G}$ and $L$ an integral lower triangular matrix with positive diagonals $\{\ell_{11},\cdots, \ell_{dd}\}$.
\end{lemma}
(For the proof, see Appendix \textsc{Proof of Lemma \ref{lem:LV}}.)

\begin{lemma}
\label{lem:see0}
An integral matrix in $\mathfrak{M}$ is the product of finite elements from $\mathfrak{G}.$
\end{lemma}
(For the proof, see Appendix \textsc{Proof of Lemma \ref{lem:see0}}.)\\

Let $\vec{r}$ be a d-dimensional row vector for which the only possible non-zero coordinates are the first $p-1$ entries
with values $1$ or $0$,
\begin{equation*}
\vec{r} \equiv [r_1,r_2,\cdots, r_{p-1}, 0, \cdots,0], r_j \in \{0,1\}.
\end{equation*}
Define matrix $L_{\vec{r}} ^{(p)}$ and $M_{\vec{r}} ^{(p)}$ by
\begin{eqnarray*}
  L_{\vec{r}} ^{(p)} &\equiv& I + \sum_{j=1} ^{p-1} r_j \Delta_{p j}, p\leq d;\\
  M_{\vec{r}} ^{(p)} &\equiv& I + \sum_{j=1} ^{p-1} r_j \Delta_{d j}, p\leq d.
\end{eqnarray*}
We have
  \begin{eqnarray*}
    (L_{\vec{r}} ^{(p)})^{-1} &=& I - \sum_{j=1} ^{p-1} r_j \Delta_{p j};\\
    I_{p d} L_{\vec{r}} ^{(p)} I_{p d} &=& M_{\vec{r}} ^{(p)}.\\
  \end{eqnarray*}

\begin{proposition}
\label{prop:see02}
Let $B$ be an integral matrix with $|\det(B)| =  2.$ Then for some natural number $p\leq d, n_0$ and $ m_0,$ and
some vector
\begin{equation*}
\vec{r} \equiv [r_1,r_2,\cdots, r_{p-1}, 0, \cdots,0], r_j \in \{0,1\}.
\end{equation*} we have
\begin{equation}
\label{eq:prop3.1eq1}
  B = L_{\vec{r}} ^{(p)} D_p V.
\end{equation}
\begin{equation}
\label{eq:prop3.1eq2}
  B = I_{p d} M_{\vec{r}} ^{(p)} D_d U.
\end{equation}
\begin{equation}
\label{eq:prop3.1eq3}
  B = V_1 V_2 \cdots V_{n_0} D_d U_1 U_2 \cdots U_{m_0}.
\end{equation}

where $V$ and $U$ are matrices in $\mathfrak{M}$, $V_j, U_i \in \mathfrak{G}$,
$L_{\vec{r}} ^{(p)} $ and
$M_{\vec{r}} ^{(p)}$ are as defined as before Proposition \ref{prop:see02}
\end{proposition}

\begin{proof}
 Let $B$ be an integral matrix with $\det B= \pm 2.$ By Lemma \ref{lem:LV}, we have $B=LU$
where $L = (\ell_{i j})$ is lower triangular and $U\in \mathfrak{M}$. So $\ell_{i j}=0,$ if $i<j$ and $\ell_{i i} >0, i=1, \cdots d,$
and $\det(L) = \Pi \ell_{jj} =2.$ Therefore, every diagonal element has value $1$ except for $\ell_{p p}=2$ for some
$p\leq d.$

By the methods we used in proof of Lemma \ref{lem:see0}, (we do all  right multiplication  but $i=p.$) we can factor
$B=LU=L_1 V_1 $ where $V_1$ is the product of some finite elements in $\mathfrak{G}$ and
$L_1 =(c_{i j})$ is lower triangular, and its $p^{\textrm{th}}$-row is
$[c_{p 1}\ c_{p 2} \cdots c_{p (p-1)} \ 2  \ 0  \cdots 0]$ and all entries in other rows are the same
entries in corresponding positions as in identity matrix $I$. For $j=1, \cdots, p-1,$ $c_{p j} =2 m_j+ r_{j}, 0 \leq r_j\leq 1.$
We define the row vector $\vec{r}$
\begin{equation*}
\vec{r} \equiv [r_{p 1}\ r_{p 2} \cdots r_{p (p-1)} \ 2  \ 0  \cdots 0].
\end{equation*}
Notice that the inverse to $(I-\Delta_{p j})^{m_{j} }$ is $(I+\Delta_{p j})^{m_{j} },$ we have
\begin{eqnarray*}
  B &=& L_1 V_1 = \big(L_1   (I-\Delta_{p (p-1)})^{m_{(p-1)} } \big)
    \big( (I+\Delta_{p (p-1)})^{m_{(p-1)} }  V_1 \big) \\
   &=&  \big(L_1   (I-\Delta_{p (p-1)})^{m_{(p-1)} }
   \cdots  (I-\Delta_{p 2})^{m_{2} } \big) \cdot\\
   &&  \big((I+\Delta_{p 2})^{m_{2} }  \cdots (I+\Delta_{p (p-1)})^{m_{(p-1)} }  V_1 \big)\\
   &=&  L_0 V, \textrm{ where }\\
   L_0 &\equiv& L_1   (I-\Delta_{p (p-1)})^{m_{(p-1)} }
   \cdots  (I-\Delta_{p 2})^{m_{2} },\textrm{ and }\\
   V &\equiv& (I+\Delta_{p 2})^{m_{2} }  \cdots (I+\Delta_{p (p-1)})^{m_{(p-1)} }  V_1.
\end{eqnarray*}
It is clear that $V$ is the product of finite elements from $\mathfrak{G}$ and it is in $\mathfrak{M}$.
Also, for $j<p,$ the only different between matrices $L_1$ and $L_1 (I-\Delta_{p 2})^{m_{2} }$ is,
at $pj$ position the first matrix has value $c_{p j}$  and the second (product) matrix has value $r_j.$
So $L_0$ is a lower triangular matrix for which if $j\neq p$ then the $j^{th}$ row vector is $\vec{e}_j$
and the $p^{th}$ row vector is $\vec{r}+2\vec{e}_p.$ So, $L_0=L_{\vec{r}} ^{(p)} D_p.$
Hence we have $B=L_{\vec{r}} ^{(p)} D_p V. $ Since $I_{d p}^2 =I,$ we have
\begin{equation*}
    B = L_{\vec{r}} ^{(p)} D_p V =  I_{d p} (I_{d p} L_{\vec{r}} ^{(p)} I_{d p}) ( I_{d p} D_p I_{d p}) (I_{d p} V)=I_{d p} M_{\vec{r}} ^{(p)}D_d U.
\end{equation*}
Here $U\equiv I_{d p} V$ is the finite product of elements from $\mathfrak{G}$. Its is easy to verify that $ I_{d p} D_p I_{d p} = D_d$ and
$I_{d p} L_{\vec{r}} ^{(p)} I_{d p} = M_{\vec{r}} ^{(p)}.$

\end{proof}

\section{Partition Theorem}\label{ss:s2}

The purpose for this section is to establish the following Theorem \ref{theom:that}. We call it Partition Theorem.
The partition properties are essential in our approach in this paper.

\begin{theorem}
\label{theom:that}
Every integral matrix $B$  with $|\det(B)| = 2$ is integrally similar to an integral matrix $A$ with the properties that
\begin{enumerate}
  \item
\begin{equation}
\label{eq:(1)}
    A \Z^d = A^\tau \Z^d.
\end{equation}
  \item There exists a vector $\vec{\ell}_A\in\Z^d$ such that
\begin{equation}
\label{eq:(2)}
 \Z^d = (\vec{\ell}_A + A \Z^d) \dotcup A \Z^d.
\end{equation}
  \item There exists a vector $\vec{q}_A \in \Z^d$
\begin{equation}\label{eq:(3)}
  \vec{q}_A \circ A \Z^d \subseteq 2\Z
 \textit{ { and \ }}
        \vec{q}_A \circ (\vec{\ell}_A + A \Z^d) \subseteq 2\Z+1.
\end{equation}
  \item For $\vec{m}\in A\Z^d,$ we have
\begin{equation}
\label{eq:partition}
    \Z^d = (\vec{n}-A\Z^d) \dotcup (\vec{\ell}_A-\vec{m}-\vec{n}+A\Z^d), \ \forall \vec{n}\in\Z^d.
\end{equation}
  \item For $\vec{m}\in \vec{\ell}_A+A\Z^d$, we have
\begin{equation}\label{eq:eq}
    \vec{n}-A\Z^d = \vec{\ell}_A-\vec{m}-\vec{n}+A\Z^d, \ \forall \vec{n}\in\Z^d.
\end{equation}
\end{enumerate}
\end{theorem}

\begin{remark}
Note that the matrices in Theorem \ref{theom:that} are not necessarily expansive. For example, for $d\geq 2,$ $D_d$ is not expansive but determinant value is $2$.
\end{remark}

\begin{proof}

By Proposition \ref{prop:see02} we have
\begin{eqnarray*}
  B &=& I_{p d} M_{\vec{r}} ^{(p)} D_d U_1
\end{eqnarray*}
for some $U_1\in \mathfrak{M}$, $p\leq d$ and $\vec{r}= [r_1,r_2,\cdots, r_{p-1}, 0, \cdots,0], \  r_j \in \{0,1\}.$ Define
\begin{equation*}
   C_1 \equiv  U_1 B U_1 ^{-1} = U_1 I_{p d} M_{\vec{r}} ^{(p)} D_d.
\end{equation*}
We have $|\det(C_1)|= 2$. By Proposition \ref{prop:see02} again,
\begin{eqnarray*}
 C_1 &=& I_{q d} M_{\vec{s}} ^{(q)} D_d U_2
 \end{eqnarray*}
for some $U_2\in \mathfrak{M}$, $q\leq d$ and $\vec{s}= [s_1,s_2,\cdots, s_{q-1}, 0, \cdots,0], s_j \in \{0,1\}.$
Define $C_2 \equiv
  (I_{q d} M_{\vec{s}} ^{(q)})^{-1} C_1 (I_{q d} M_{\vec{s}} ^{(q)})$.
  We have
 \begin{eqnarray*}
 C_2 &=& D_d U_2 (I_{q d} M_{\vec{s}} ^{(q)})\\
  &=&D_d W_1
 \end{eqnarray*}
where $W_1\equiv U_2 (I_{q d} M_{\vec{s}} ^{(q)})\in \mathfrak{M}.$
 Also we have
 \begin{eqnarray*}
 C_2 &=&      (I_{q d} M_{\vec{s}} ^{(q)})^{-1} U_1 I_{p d} M_{\vec{r}} ^{(p)} D_d (I_{q d} M_{\vec{s}} ^{(q)})\\
 &=& W_2 D_d I_{q d} M_{\vec{s}} ^{(q)}
  \end{eqnarray*}
where $W_2 \equiv (I_{q d} M_{\vec{s}} ^{(q)})^{-1} U_1 I_{p d} M_{\vec{r}} ^{(p)} \in \mathfrak{M}.$
Define a matrix $X$ be
\begin{equation}
\label{eq:X}
X \equiv  \left\{
\begin{array}{ll}
I \quad &\texttt{ if }q=d, \\
I-\Delta_{d q} \quad &\texttt{ if } q< d.
\end{array} \right.
\end{equation}
Define
$A\equiv XC_2 X^{-1}.$ It is clear that the matrices $A$ and $B$ are integrally similar.
We claim that the matrix $A$ has the desired properties (1), (2), (3), (4) and (5).

We have
\begin{eqnarray*}
  A &=& XD_d W_1 X^{-1}, \\
  A &=& XW_2 D_d I_{q d} M_{\vec{s}} ^{(q)} X^{-1},
\end{eqnarray*}
then we have
\begin{equation}
\label{eq:A^tau} A^\tau
  = (X^{-1})^\tau (M_{\vec{s}} ^{(q)})^\tau I_{q d}D_d (W_2)^\tau X^\tau,
\end{equation}
and
\begin{equation}
  \label{eq:A^{-1}} A^{-1}
  = X (W_1)^{-1} (D_d)^{-1} X^{-1}.
\end{equation}

Proof of Property (1).
 It is clear that $|\det(A^{-1} A^\tau)|=1.$ We will show that the product $A^{-1} A^\tau$
is an integral matrix. Then we will have $A^{-1} A^\tau \Z^d = \Z^d.$ This is equivalent to the equation $A^\tau \Z^d = A \Z^d.$ We have two cases.

Case (A) First we assume $q<d.$
By \eqref{eq:A^{-1}} and \eqref{eq:A^tau} we have
\begin{eqnarray*}
  A^{-1} A^\tau
  &=& X (W_1)^{-1} \big[ (D_d)^{-1} X^{-1}    (X^{-1})^\tau (M_{\vec{s}} ^{(q)})^\tau I_{q d}D_d \big] (W_2)^\tau X^\tau
\end{eqnarray*}
$(D_d)^{-1}$ and $D_d$ are the only matrices on the right product which are not in $\mathfrak{M}.$  We will prove that the
product $(D_d)^{-1} X^{-1}   (X^{-1})^\tau (M_{\vec{s}} ^{(q)})^\tau I_{q d}D_d$ is an integral matrix. This will complete the proof of (1).

It is clear that $X^{-1} = I+\Delta_{d q}$ and $(X^{-1})^\tau = I+\Delta_{q d}$ and $(D_d)^{-1} = I - \frac{1}{2}\Delta_{d d}.$
So we have
\begin{eqnarray*}
  (D_d)^{-1} X^{-1} \cdot (X^{-1})^\tau &=& (I - \frac{1}{2}\Delta_{d d}) \cdot (I+\Delta_{d q})\cdot (I+\Delta_{q d})\\
  &=& (I - \frac{1}{2}\Delta_{d d}) \cdot (I + \Delta_{q d} + \Delta_{d q} +\Delta_{d d})\\
  &=& I + \Delta_{q d} +\frac{1}{2} \Delta_{d q}.
\end{eqnarray*}
Here we used the Equation \eqref{eq:magic}.
We also have
\begin{eqnarray*}
(M_{\vec{s}} ^{(q)})^\tau I_{q d} &=&  (I + \sum_{j=1} ^{q-1} s_j \Delta_{d j})^\tau \cdot I_{q d}\\
&=& (I + \sum_{j=1} ^{q-1} s_j \Delta_{j d }) \cdot I_{q d}\\
&=& I_{q d} + \sum_{j=1} ^{q-1} s_j \Delta_{j q }.
\end{eqnarray*}
Repeatedly using Equation \eqref{eq:magic}, we obtain that
\begin{eqnarray*}
  A^{-1} A^\tau &=&
\big( (D_d)^{-1} X^{-1} \cdot (X^{-1})^\tau \big) \cdot
\big(  (M_{\vec{s}} ^{(q)})^\tau I_{q d} \big)
   \cdot D_d\\
 &=& (I + \Delta_{q d} +\frac{1}{2} \Delta_{d q}) \cdot (I_{q d} + \sum_{j=1} ^{q-1} s_j \Delta_{j q }) \cdot D_d\\
 &=& (\frac{1}{2}\Delta_{d d} + I_{q d} + \sum_{j=1} ^{q-1} s_j \Delta_{j q }) \cdot D_d\\
 &=& \Delta_{d d}+2 I_{q d} + \sum_{j=1} ^{q-1} s_j \Delta_{j q } .
\end{eqnarray*}
This is an integral matrix.
In the above last step  we used the following equality
\begin{equation}
\Delta_{i j} \cdot D_d =  \left\{
\begin{array}{ll}
2 \Delta_{i j} \quad & j=d, \\
\Delta_{i j} \quad & j< d.
\end{array} \right.
\end{equation}


Case (B)
Next, we assume $q=d.$ By definition \eqref{eq:X} $X=I.$ We have
$A = C_2 = D_d W_1 =W_2 D_d  M_{\vec{s}} ^{(d)}.$
So, $A^{-1}
  = (W_1)^{-1} (D_d)^{-1} $ and
  $A^\tau = (M_{\vec{s}} ^{(d)})^\tau D_d (W_2)^\tau.$
\begin{eqnarray*}
  A^{-1} A^\tau
  &=& W_1 ^{-1} (D_d)^{-1} (M_{\vec{s}} ^{(d)})^\tau D_d  W_2 ^\tau.
  \end{eqnarray*}
Since $ W_1 ^{-1}$ and $W_2 ^\tau$ are integral matrices, it is enough to prove that the product
$(D_d)^{-1} (M_{\vec{s}} ^{(d)})^\tau D_d$ is integral.
  \begin{eqnarray*}
  (D_d)^{-1} (M_{\vec{s}} ^{(d)})^\tau D_d
  &=& (I - \frac{1}{2}\Delta_{d d})\cdot (I + \sum_{j=1} ^{d-1} s_j \Delta_{d j})^\tau \cdot D_d\\
  &=& (I - \frac{1}{2}\Delta_{d d})\cdot (I + \sum_{j=1} ^{d-1} s_j \Delta_{j d}) \cdot D_d\\
  &=& (I - \frac{1}{2}\Delta_{d d} + \sum_{j=1} ^{d-1} s_j \Delta_{j d}) \cdot D_d\\
  &=& I  + \sum_{j=1} ^{d-1} 2 s_j \Delta_{j d}.
\end{eqnarray*}
So, $A^{-1} A^\tau$ an integral matrix.  Proof of Property (1) is complete.\\

Proof of Property (2). Since $|\det(A)|=2,$ by Proposition \ref{prop:see02} we have
\begin{equation}
\label{eq:ok}
    A = I_{p_0 d} M_{\vec{r}_0} ^{(p_0)} D_d U_0
\end{equation}
for some $p_0 \leq d,$ $U_0 \in \mathfrak{M} $ and vector
$\vec{r}_0 \equiv \sum_{j=1} ^{p_0 -1} \varrho_j \vec{e}_j, \varrho_j\in \{0,1\}.$ This is a vector in $\Z^d.$ Also we have
\begin{equation}
\label{eq:mq}
   M_{\vec{r}_0} ^{(p_0)} = I + \sum_{j=1} ^{p_0 -1} \varrho_j \Delta_{d j}.
\end{equation}
Let $\vec{v}=\sum_j x_j \vec{e}_j$ be a vector in $\Z^d.$ The coefficient $x_d$ is even or odd.
So we have
\begin{eqnarray*}
\Z^d &=& (\vec{e}_d + D_d \Z^d ) \dotcup D_d \Z^d .
\end{eqnarray*}
All matrices $I_{p_0 d}, M_{\vec{r}_0} ^{(p_0)}$ and $U_0$ are in $\mathfrak{M}.$ So
$U_0 \Z^d=\Z^d$ and
$I_{p_0 d} M_{\vec{r}_0} ^{(p_0)} \Z^d = \Z^d.$ We have
\begin{eqnarray*}
\Z^d &=& I_{p_0 d} M_{\vec{r}_0} ^{(p_0)} \Z^d\\
&=& I_{p_0 d} M_{\vec{r}_0} ^{(p_0)} \big( (\vec{e}_d + D_d \Z^d ) \dotcup D_d \Z^d\big)\\
&=& I_{p_0 d} M_{\vec{r}_0} ^{(p_0)} \big( (\vec{e}_d + D_d U_0 \Z^d ) \dotcup D_d U_0 \Z^d\big)
\end{eqnarray*}
Define
\begin{equation*}
  \vec{\ell}_A \equiv I_{p_0 d} M_{\vec{r}_0} ^{(p_0)} \vec{e}_d=  \vec{e}_{p_0}.
\end{equation*}
By Equation \eqref{eq:ok} we obtain
\begin{eqnarray*}
 \Z^d &=& (I_{p_0 d} M_{\vec{r}_0} ^{(p_0)} e_d +I_{p_0 d} M_{\vec{r}_0} ^{(p_0)}D_d U_0 \Z^d)\dotcup I_{p_0 d} M_{\vec{r}_0} ^{(p_0)}D_d U_0 \Z^d\\
&=& (\vec{\ell}_A + A \Z^d) \dotcup A \Z^d.
\end{eqnarray*}

Proof of Property (3). Case (A). Assume $p_0 <d.$
We will say that $A \Z^d$ is integrally generated by a (finite) set $\{\vec{u}_j\}$ if for each vector $\vec{v}\in A \Z^d$ there exists a (finite) set of integers $\{s_j\}$ such that $\vec{v} = \sum s_j \vec{u}_j.$ We will call the set a generator for $A \Z^d.$
The set $\{ A e_j, j=1,\cdots, d\}$ is a generator for $A \Z^d.$
By Equation \eqref{eq:ok} $A = I_{p_0 d} M_{\vec{r}_0} ^{(p_0)}D_d U_0.$ Since $U_0 \Z^d = \Z^d,$
the set $\{I_{p_0 d} M_{\vec{r}_0} ^{(p_0)}D_d \vec{e}_i,i=1,\cdots, d\}$ is a generator for $A \Z^d.$
We have
\begin{eqnarray*}
I_{p_0 d}  M_{\vec{r}_0} ^{(p_0)} D_d \vec{e}_i
&=& I_{p_0 d} \big(I + \sum_{j=1} ^{p_0 -1} \varrho_j \Delta_{d j}\big) D_d \vec{e}_i\\
&=& \left\{
\begin{array}{lll}
I_{p_0 d}\big((1+\varrho_i) \vec{e}_i\big) \quad & \texttt{ if }i\leq p_0-1, \\
I_{p_0 d}(\vec{e}_{p_0}) \quad & \texttt{ if } i = p_0,\\
I_{p_0 d}(\vec{e}_i) \quad & \texttt{ if }p_0+1 \leq i \leq d-1,\\
I_{p_0 d}(2\vec{e}_d) \quad & \texttt{ if }i=d.
\end{array} \right.\\
&=&
   \left\{
\begin{array}{llll}
(1+\varrho_i) \vec{e}_i \quad & \texttt{ if }i\leq p_0-1, \\
\vec{e}_d \quad & \texttt{ if } i=p_0,\\
\vec{e}_i \quad & \texttt{ if } p_0+1 \leq i \leq d-1,\\
2\vec{e}_{p_0} \quad & \texttt{ if }i=d.
\end{array} \right.
\end{eqnarray*}

Define
\begin{equation}
\label{eq:theq}
\vec{q}_A\equiv \vec{e}_{p_0}+\sum_{j=1} ^{p_0 -1} \varrho_j \vec{e}_j.
\end{equation}
To prove that $\vec{q}_A \circ A \Z^d \subseteq 2\Z,$ it is enough to show that
$\vec{q}_A\circ (I_{p_0 d}  M_{\vec{r}_0} ^{(p_0)} D_d \vec{e}_i )$ is an even integer for each $i\leq d.$ We have
\begin{eqnarray*}
 &&\vec{q}_A\circ (I_{p_0 d}  M_{\vec{r}_0} ^{(p_0)} D_d \vec{e}_i )  \\
 &=&
   \left\{
\begin{array}{llll}
\vec{q}_A\circ \big((1+\varrho_i) \vec{e}_i \big) \quad & \texttt{ if }i\leq p_0-1, \\
\vec{q}_A\circ \vec{e}_d \quad & \texttt{ if } i=p_0,\\
\vec{q}_A\circ \vec{e}_i \quad & \texttt{ if } p_0+1 \leq i \leq d-1,\\
\vec{q}_A\circ (2\vec{e}_{p_0}) \quad & \texttt{ if }i=d.
\end{array} \right.\\
 &=&
   \left\{
\begin{array}{llll}
(\vec{e}_{p_0}+\sum_{j=1} ^{p_0 -1} \varrho_j \vec{e}_j)
\circ \big((1+\varrho_i) \vec{e}_i \big) \quad & \texttt{ if }i\leq p_0-1, \\
( \vec{e}_{p_0}+\sum_{j=1} ^{p_0 -1} \varrho_j \vec{e}_j)
\circ \vec{e}_d \quad & \texttt{ if } i=p_0,\\
( \vec{e}_{p_0}+\sum_{j=1} ^{p_0 -1} \varrho_j \vec{e}_j)
\circ \vec{e}_i \quad & \texttt{ if } p_0+1 \leq i \leq d-1,\\
( \vec{e}_{p_0}+\sum_{j=1} ^{p_0 -1} \varrho_j \vec{e}_j)
\circ (2\vec{e}_{p_0}) \quad & \texttt{ if }i=d.
\end{array} \right.\\
 &=&
   \left\{
\begin{array}{llll}
\varrho_i(1+\varrho_i) \quad & \texttt{ if }i\leq p_0-1, \\
0 \quad & \texttt{ if } i=p_0,\\
0 \quad & \texttt{ if } p_0+1 \leq i \leq d-1,\\
2 \quad & \texttt{ if }i=d.
\end{array} \right.\\
\end{eqnarray*}
The values of all above inner products are even integers.
So we proved that when $p_0 <d,$ $\vec{q}_A \circ A \Z^d$ is a set of even integers. Also,
\begin{equation*}
    \vec{q}_A \circ \vec{\ell}_A =
    \big(\vec{e}_{p_0}+\sum_{j=1} ^{p_0 -1} \varrho_j \vec{e}_j \big) \circ
    \vec{e}_{p_0} =1.
\end{equation*}
This implies that $\vec{q}_A \circ (\vec{\ell}_A + A \Z^d)$ is a set of odd integers.\\

Case (B) Now we assume $p_0=d.$ Then $I_{p_0 d} =I.$
In this case
\begin{eqnarray*}
  A &=& M_{\vec{r}_0} ^{(d)} D_d U_0,\\
  \vec{r}_0 &=& \sum_{j=1} ^{d-1} \varrho_j e_j, \varrho_j\in \{0,1\}.
\end{eqnarray*}

By the definition in Proof of Property (2)
\begin{eqnarray*}
  \vec{\ell}_A &=& \vec{e}_d.
\end{eqnarray*}
Define
\begin{equation*}
    \vec{q}_A\equiv \vec{e}_{d}+\sum_{j=1} ^{d -1} \varrho_j \vec{e}_j.
\end{equation*}
It is clear that $\vec{q}_A \circ \vec{\ell}_A = (\vec{e}_{d}+\sum_{j=1} ^{d -1} \varrho_j \vec{e}_j) \circ \vec{e}_d = 1.$ Denote $\vec{u}_i \equiv  M_{\vec{r}_0} ^{(d)} D_d \vec{e}_i, i=1,\cdots,d.$ Then the set $\{\vec{u}_i, i=1,\cdots, d\}$ is a generator for $A \Z^d.$ To complete the proof of Case (B), it is enough to show that $\vec{q}_A \circ \vec{u}_i$ is even integer for each $i\leq d.$ We have
\begin{eqnarray*}
  \vec{u}_i &=& M_{\vec{r}_0} ^{(d)} D_d \vec{e}_i, i=1,\cdots,d \\
&=& \left\{
\begin{array}{lll}
M_{\vec{r}_0} ^{(d)} \vec{e}_i \quad & \texttt{ if }i\leq d-1, \\
2M_{\vec{r}_0} ^{(d)}\vec{e}_d \quad & \texttt{ if }i=d.
\end{array} \right.\\
&=& \left\{
\begin{array}{lll}
(1+\varrho_i)\vec{e}_i \quad & \texttt{ if }i\leq d-1, \\
2\vec{e}_d \quad & \texttt{ if }i=d.
\end{array} \right.\\
\end{eqnarray*}
We have
\begin{eqnarray*}
  \vec{q}_A \circ \vec{u}_i &=& \left\{
\begin{array}{lll}
(\vec{e}_{d}+\sum_{j=1} ^{d -1} \varrho_j \vec{e}_j)\circ\big((1+\varrho_i)\vec{e}_i\big) \quad & \texttt{ if }i\leq d-1, \\
(\vec{e}_{d}+\sum_{j=1} ^{d -1} \varrho_j \vec{e}_j)\circ
(2\vec{e}_d) \quad & \texttt{ if }i=d.
\end{array} \right.\\
&=& \left\{
\begin{array}{lll}
(1+\varrho_i) \varrho_i \quad & \texttt{ if }i\leq d-1, \\
2 \quad & \texttt{ if }i=d.
\end{array} \right.
\end{eqnarray*}
Apparently, those are even numbers. Case (B) has been proved.

Proof of Property (4). Let $\vec{m}\in A\Z^d$. Since $A\Z^d$ is a group containing $(2\Z)^d,$ we have
$-A\Z^d=A\Z^d$ and  $-\vec{m}-\vec{n}+A\Z^d= \vec{n}+A\Z^d$. So
\begin{eqnarray}
\label{eq:(a)}
(\vec{n}-A\Z^d) \cup (\vec{\ell}_A-\vec{m}-\vec{n}+A\Z^d)
   &=&
(\vec{n}+A\Z^d) \cup (\vec{\ell}_A+\vec{n}+A\Z^d) .
\end{eqnarray}
If $\vec{n}\in A\Z^d$ then $\vec{n}+A\Z^d = A\Z^d$ and
$\vec{\ell}_A + \vec{n}+A\Z^d = \vec{\ell}_A +A \Z^d.$
If $\vec{n}\in \vec{\ell}_A+A\Z^d$ then $\vec{n}+A \Z^d =\vec{\ell}_A+ A\Z^d$ and
$\vec{\ell}_A + \vec{n}+A\Z^d = A \Z^d.$ So, in either cases we have
\begin{equation*}
(\vec{n}+A\Z^d) \cup (\vec{\ell}_A+\vec{n}+A\Z^d)
  =
(A\Z^d) \dotcup (\vec{\ell}_A+A\Z^d)=\Z^d.
\end{equation*}

Proof of Preperty (5) Let $\vec{m}\in \vec{\ell}_A+A \Z^d.$ Then $\vec{\ell}_A-\vec{m}\in A\Z^d = A\Z^d.$
So $\vec{\ell}_A-\vec{m}-\vec{n}+A\Z^d=\vec{n} - A\Z^d.$ This is \eqref{eq:eq}.

\end{proof}


\section{Reduction Theorem}

\bigskip

For $f,g\in L^1(\R^d) \cap L^2(\R^d),$ the Fourier-Plancherel Transform and Fourier Inverse Transform are defined as
\begin{eqnarray*}
(\F f)(\vec{s}) &\equiv& \frac{1}{(2\pi)^{d/2}} \int_{\R^d}e^{-i\vec{s}\circ\vec{t}}f(\vec{t})d\vt=\hat{f}(\vec{s}),\\
(\F^{-1} g)(\vec{t})&\equiv& \frac{1}{(2\pi)^{d/2}}\int_{\R^d}e^{i\vec{s}\circ\vec{t}}g(\vec{s})d\vec{s}=\check{g}(\vec{t}).\\
\end{eqnarray*}
The set $L^1(\R^d) \cap L^2(\R^d)$ is dense in $L^2(\R^d),$ the operator $\F$ extends to a unitary operator on $L^2(\R^d)$ which is still called Fourier Transform.
For an operator $V$ on $L^2(\R^d),$ we will use notation $\widehat{V}\equiv\F V \F^{-1}.$
We collect the following elementary equalities in Lemma \ref{lem:noproof}. We omit the proof.
\begin{lemma}
\label{lem:noproof}
Let $A$ be a $d\times d$ expansive integral matrix with $|\det(A)| =2$; $\vec{t},\vec{s}$ and $\vec{\ell}\in \R^d$
and $J\in\Z$.
For a $d\times d$ integral matrix $S$ of $|\det(S)|=1,$ assume $B=S^{-1}AS$.
Then
\begin{eqnarray}
\label{eq:P} U_S T_{\vec{\ell}}U_S^{-1} &=& T_{S^{-1}\vec{\ell}}.\\
\label{eq:Q} U_S D_A ^J U_S^{-1} &=& D_B ^J.\\
\label{eq:R}  T_{\vec{\ell}} D_A &=& D_AT_{A\vec{\ell}}.\\
\label{eq:multiplier}  \widehat{T}_{\vec{\ell}} &=& M_{e^{- i \vec{s}\circ\vec{\ell}}}.\\
  \widehat{D}_A &=& U_{{(A^{-1})}^\tau} = U_{{(A^\tau)}^{-1}}=D_{A^\tau}^{-1}=D_{A^\tau}^*.\\
\label{eq:3.6}  T_{A^{-J}\vec{\ell}} D_A^{J}&= &D_A ^J T_{\vec{\ell}}.\\
\label{eq:3.7}  \overline{\widehat{D}_A^{J}\widehat{\phi}(\vt)} &=&
\frac{1}{\sqrt{2^J}}\overline{\widehat{\phi}((A^\tau)^{-J}\vt)}, \phi\in L^2(\R^d).
\end{eqnarray}
where $M_{e^{- i \vec{s}\circ\vec{\ell}}}$ is the unitary multiplication operator by $e^{- i \vec{s}\circ\vec{\ell}}$.
\end{lemma}

\begin{theorem}
\label{theom:redsym}
Let $A$ be a $d\times d$ expansive integral matrix with $|\det(A)| =2$ and $S$ be a $d\times d$ integral matrix with property  that $|\det(S)|=1$. Let $B\equiv S^{-1}AS$.  Assume that
a function $\psi_A$ is a Parseval's frame wavelet associated with the matrix $A.$ Then the function $\eta_B \equiv U_S \psi_A$ is a Parseval's frame wavelet associated with the matrix $B.$
\end{theorem}
(For the proof, see Appendix \textsc{Proof of Theorem \ref{theom:redsym}}).\\

\section{ Lawton's Equations and Filter Function $m_0(\vec{t})$}
\bigskip

Through out the rest of this paper, $A$ will be a $d\times d$ expansive integral matrix with $|\det (A)|=2$ satisfying the equations \eqref{eq:(1)},\eqref{eq:(2)},\eqref{eq:(3)}, \eqref{eq:partition} and \eqref{eq:eq}
in Partition Theorem.
Let $\mathcal{S}=\{h_{\vec{n}}:~{\vec{n}\in\Z^d}\}$ be a \textit{finite} complex solution to the system of equations
\begin{equation}
\label{eq:lawtoneq}
\left\{\begin{array}{l}
\sum_{\vec{n}\in\Z^d}h_{\vec{n}}\overline{h_{\vec{n}+\vec{k}}}=\delta_{\vec{0} \vec{k}},~ \vec{k}\in A\Z^d \\
\sum_{\vec{n}\in\Z^d}h_{\vec{n}}=\sqrt{2}.
\end{array}\right.
\end{equation}
Here  $\delta$ is  the Kronecker delta.
The solution is finite if  $h_{\vec{n}}=0$ for all
$\vec{n}\notin \Lambda_0 \equiv \Z^d \cap [-N_0,N_0]^d$ for some $N_0 \in\N$.
Wayne Lawton introduced \cite{lawton} a system of equations for Parseval's frame wavelet in $L^2(\R).$ The equations \eqref{eq:lawtoneq} are its generalization in higher dimensions.
We will call the system of equations
\eqref{eq:lawtoneq}
\textit{Lawton's system of equations for Parseval's frame wavelets associated with matrix $A$}, or \textit{Lawton's equations associated with matrix $A$}.

Define
\begin{equation}
\label{eq:m0}
    m_0(\vec{t}) \equiv \frac{1}{\sqrt{2}}\sum_{\vec{n}\in\Z^d} h_{\vec{n}} e^{-i\vec{n}\circ\vec{t}}
    =\frac{1}{\sqrt{2}}\sum_{\vec{n}\in\Lambda_0} h_{\vec{n}} e^{-i\vec{n}\circ\vec{t}},\ \vec{t} \in \C^d.
\end{equation}
This is a finite sum and $m_0(\vec{0})=1$. It is a $2\pi$-periodic trigonometric polynomial function in the sense that $m_0(\vec{t})=m_0(\vec{t}+\pi\vec{t}_0), \ \forall \vec{t}_0\in (2\Z)^d.$

\begin{proposition}
\label{prop:filter}
Let $A$ be an expansive $d \times d$ integral matrix and $\vec{q}_A  $ is as stated in
Partition Theorem.
Let $m_0(\vec{t})$ be  defined as in \eqref{eq:m0}, then
\begin{equation}\label{m0eq}
 |m_0(\vec{t})|^2+|m_0(\vec{t}+\pi \vec{q}_A)|^2=1, \ \forall \vec{t}\in\R^d.
 \end{equation}
\end{proposition}
\begin{corollary}\label{less1}
\begin{equation}
|m_0(\vec{t})|\leq 1,\ \forall\vec{t}\in\R^d.
\end{equation}
\end{corollary}
(For the proof, see Appendix \textsc{Proof of Proposition \ref{prop:filter}}.)\\

\section{The Scaling Function $\varphi$ in $L^2(\R^d)$}\label{ss:phil2}
\bigskip

\begin{definition}
\label{def:g}
Define
 \begin{equation*}
 g(\vec{\xi}) \equiv \frac{1}{(2\pi)^{d/2}}\prod_{j=1}^{\infty}m_0((A^{\tau})^{-j} \vec{\xi}), \forall\vec{\xi}\in\R^d,
\end{equation*}
and
\begin{equation}
\label{def:phi}
\varphi = \F^{-1} g.
\end{equation}
\end{definition}

In this section we will prove that $g$  is an $L^2(\R^d)$-function. Hence by Plancherel Theorem $\varphi$ is also in $L^2(\R^d).$ This provides the base for further construction in this paper. The proof we used in \cite{dai2d} (for case $d=2$) no longer works for our general cases here. We will call $\varphi$ the \textit{scaling function}.
We will use $\Gamma_\pi$ to denote $[-\pi,\pi)^d.$\\
\begin{lemma}
\label{lem:sesami}
Let $f$ be a $2\pi$-periodical continuous function on $\R^d$ and let $\vec{\gamma}\in \R^d.$ Then
\begin{equation*}
    \int_{\Gamma_\pi + \vec{\gamma}} f(\vt) d \vt =
    \int_{\Gamma_\pi } f(\vt) d \vt.
\end{equation*}
\end{lemma}

\begin{proof}
Denote $\Gamma_\pi ^{\mbox{\tiny ($ \vec{n}) $ }} \equiv \Gamma_\pi + \pi \vec{n}, \ \vec{n}\in\Z^d.$  The family
$\{\Gamma_\pi ^{\mbox{\tiny ($ \vec{n}) $ }}, \vec{n}\in(2\Z)^d\}$ is a partition of $\R^d.$
Since $\Gamma_\pi + \vec{\gamma}$ is bounded there is a finite subset
$\vec{\Lambda} \subset (2\Z)^d,$
\begin{equation*}
    \bigdotcup _{\vec{n}\in \vec{\Lambda}} (\Gamma_\pi + \vec{\gamma})\cap \Gamma_\pi ^{\mbox{\tiny ($ \vec{n}) $ }} = \Gamma_\pi + \vec{\gamma}.
\end{equation*}
Denote $\Gamma_{\pi, {\mbox{\tiny $ \vec{n} $ }}} \equiv (\Gamma_\pi + \vec{\gamma})\cap \Gamma_\pi ^{\mbox{\tiny ($ \vec{n}) $ }} - \pi \vec{n}.$ We claim that the family $\{\Gamma_{\pi, {\mbox{\tiny $ \vec{n} $ }}}, \vec{n}\in\vec{\Lambda}\}$ is a partition of $\Gamma_\pi.$ It is clear that
$\Gamma_{\pi, {\mbox{\tiny $ \vec{n} $ }}} \subseteq \Gamma_\pi, \ \forall \vec{n} \in \vec{\Lambda}.$
Assume that for some different $\vec{n}_1,\vec{n}_2 \in \vec{\Lambda},$
there is a vector
$\vec{u}\in\Gamma_{\pi, {\mbox{\tiny $ \vec{n}_1 $ }}} \cap
\Gamma_{\pi, {\mbox{\tiny $ \vec{n}_2 $ }}} . $
This implies that $\vec{u}+\pi\vec{n}_1 \in \Gamma_\pi +\vec{\gamma}$ and
$\vec{u}+\pi\vec{n}_2 \in \Gamma_\pi +\vec{\gamma},$ or
$\vec{u}-\vec{\gamma}\in \Gamma_\pi ^{\mbox{\tiny ($ -\vec{n}_1) $ }} \cap
\Gamma_\pi ^{\mbox{\tiny ($ -\vec{n}_2) $ }}.$ This is impossible since
$\{\Gamma_\pi ^{\mbox{\tiny ($ \vec{n}) $ }}, \vec{n}\in(2\Z)^d\}$ is a partition of $\R^d.$

Now we have
\begin{eqnarray*}
  \int_{\Gamma_\pi } f(\vt) d \vt
&=& \int_{\dotcup_{\vec{n}\in\vec{\Lambda}} \Gamma_{\pi, {\mbox{\tiny $ \vec{n} $ }}}} f(\vt) d \vt\\
&=& \sum_{\vec{n}\in\vec{\Lambda}}
\int_{(\Gamma_\pi + \vec{\gamma})\cap \Gamma_\pi ^{\mbox{\tiny $ (\vec{n}) $ }} - \pi \vec{n}} f(\vt) d \vt\\
&=& \sum_{\vec{n}\in\vec{\Lambda}}
\int_{(\Gamma_\pi + \vec{\gamma})\cap \Gamma_\pi ^{\mbox{\tiny $ (\vec{n}) $ }}} f(\vt) d \vt\\
&=&
\int_{\dotcup_{\vec{n}\in\vec{\Lambda}}(\Gamma_\pi + \vec{\gamma})\cap \Gamma_\pi ^{\mbox{\tiny $ (\vec{n}) $ }}} f(\vt) d \vt\\
&=& \int_{\Gamma_\pi +\vec{\gamma} } f(\vt) d \vt.
\end{eqnarray*}

\end{proof}

\begin{lemma}
\label{lem:A}
Let $M$ be a $d\times d$ integral matrix in $\mathfrak{G}$ and $G$ be a $d\times d$ integral matrix with $\det(G) \neq 0$.
Then
\begin{enumerate}
  \item  There exists a finite partition of $\Gamma_\pi$, $\{\Delta_j, j\in\Lambda_M\}$ and a subset of elements of $(2\Z)^d,$ $\{\vec{n}_j,j\in\Lambda_M\}$ such that
\begin{equation}
\label{eq:Psecond}
    M\Gamma_\pi = \bigdotcup_{j\in\Lambda_M} (\Delta_j +\pi \vec{n}_j).
\end{equation}
  \item Let $f$ be a $2\pi$-periodical continuous function on $\R^d.$ Then
  \begin{equation}
  \label{eq:B}
\int_{GM \Gamma_\pi} f(\vt) d \vt =  \int_{G\Gamma_\pi} f(\vt) d \vt
\end{equation}
and
  \begin{equation}
  \label{eq:CC}
\int_{M \Gamma_\pi} f(\vt) d \vt =  \int_{\Gamma_\pi} f(\vt) d \vt
\end{equation}
\end{enumerate}
\end{lemma}
\begin{proof}  (1)
It is clear that
\begin{eqnarray}
\label{eq:I12}  I_{i j} \Gamma_{\pi} &=& \Gamma_{\pi},\ i,j\leq d, \\
  S_p \Gamma_{\pi} &=& \Gamma_{\pi},\ p\leq d ,
\end{eqnarray}
since $I_{i j}\Gamma_{\pi}$, $S_p \Gamma_{\pi}$ have the same vertices as of $\Gamma_{\pi}.$

Let $M\equiv I+\Delta_{i j}, \ i\neq j$ and $i,j \leq d.$
Since $I+\Delta_{1 2} = I_{2 j} I_{1 i} (I+\Delta_{i j})I_{1 i} I_{2 j}$, by \eqref{eq:I12}
we can assume $M=I+\Delta_{1 2}.$
The set $M \Gamma_\pi$ is a cylinder
in $d\times d$ space. We have $M\Gamma_\pi = P \times [-\pi,\pi]^{d-2}$ where $P$  is
the orthogonal project of $M\Gamma_\pi$ into the two dimensional coordinate plane $X_1 \times X_2$.
The set $P$ is a parallelogram (Figure 1 left). The set $P$ contains two disjoint triangles $\triangle_+$ and $\triangle_-.$
$P$ is the disjoint union of $\{\triangle_+, \triangle_-, P\backslash(\triangle_+ \dotcup \triangle_-)\}.$
Also $\Gamma_\pi$
 is disjoint union of $\{\triangle_+ - 2\pi \vec{e}_1, \triangle_- + 2\pi \vec{e}_1, P\backslash(\triangle_+ \dotcup \triangle_-)\}$  (Figure 1 right).

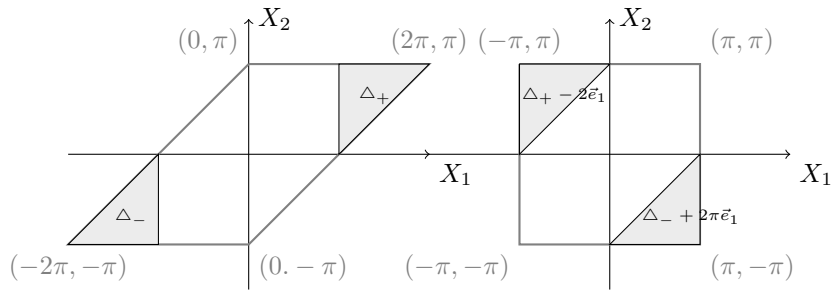
\begin{figure}[h]
\begin{tikzpicture}[scale=.6]
\draw[->] (-3.5,0)-- (12.5,0)
node[below right] {$X_1$};
\draw[->] (4.5,0)-- (4.51,0)
node[below right] {$X_1$};
\draw[->] (0.5,-3)-- (0.5,3)
node[right] {$X_2$};
\draw[->] (8.5,-3)-- (8.5,3)
node[right] {$X_2$};
\draw[step=.25cm,gray,thick] (0.5+0,2)-- (0.5+4,2)node[above] {$(2\pi,\pi)$}-- (0.5+0,-2)node[below right] {$(0.-\pi)$}-- (0.5-4,-2)node[below] {$(-2\pi,-\pi)$}-- (0.5+0,2)node[above left] {$(0,\pi)$};
\draw[step=.25cm,gray,thick] (6.5,2)-- (10.5,2)node[above right] {$(\pi,\pi)$}-- (8.5+2,-2)node[below right] {$(\pi,-\pi)$}-- (6.5,-2)node[below left] {$(-\pi,-\pi)$}-- (6.5+0,2)node[above] {$(-\pi,\pi)$};
\draw[fill=gray!15] (0.5-4,-2) -- (0.5-4+2,-2) -- (0.5-4+2,0) -- cycle;
\node at (-2-.1,-2+.6) {{\tiny$\triangle_-$}};
\draw[fill=gray!15] (0.5-4+4+8,-2) -- (0.5-4+4+8+2,-2) -- (0.5-4+4+8+2,0) -- cycle;
\node at (-2-.7+13,-2+.6) {{\tiny $\triangle_- +2\pi \vec{e}_1$}};
\draw[fill=gray!15] (6.5,0) -- (6.5,2) -- (8.5,2) -- cycle;
\draw[fill=gray!15] (2.5,0) -- (2.5,2) -- (4.5,2) -- cycle;
\node at (3.3,1+.3) {{\tiny $\triangle_+$}};
\node at (3.5+4,1+.3) {{\tiny $\triangle_+ -2\vec{e}_1$}};
\end{tikzpicture}
\caption{$X_1 \times X_2$ Coordinate Plane}
\end{figure}

Therefore
\begin{itemize}
  \item $\{\triangle_+ \times [-\pi,\pi]^{d-2},\triangle_- \times [-\pi,\pi]^{d-2},P\backslash(\triangle_+ \dotcup \triangle_-) \times [-\pi,\pi]^{d-2} \}$
is a partition of $M\Gamma_\pi$;
  \item $\{\triangle_+ \times [-\pi,\pi]^{d-2}-2\pi \vec{e}_1,\triangle_- \times [-\pi,\pi]^{d-2}+2\pi \vec{e}_1, P\backslash(\triangle_+ \dotcup \triangle_-) \times [-\pi,\pi]^{d-2} \}$
is a partition of $\Gamma_\pi$.
\end{itemize}
The proof for the case $M=I-\Delta_{i j}$ is similar to this. We omit it. \\

(2) Since $G:\R^d\rightarrow \R^d $ by $\vt\mapsto G\vt$ is linear and one-to-one, we have
\begin{eqnarray*}
GM\Gamma_\pi &=& G\big( \bigdotcup (\Delta_j +\pi \vec{n}_j) \big) = \bigdotcup (G\Delta_j +\pi G\vec{n}_j).
\end{eqnarray*}
Also, $\pi G\vec{n}_j \in\pi (2\Z)^d$ and $f$ is $2\pi$-periodical, so we have
\begin{eqnarray*}
\int_{GM \Gamma_\pi} f(\vt) d \vt
&=&   \int_{\dotcup (G\Delta_j +\pi G\vec{n}_j)} f (\vt) d \vt\\
&=&   \int_{G(\dotcup\Delta_j )} f (\vt) d \vt\\
&=&   \int_{G\Gamma_\pi} f (\vt) d \vt
\end{eqnarray*}
Replace $G$ by $I$ in \eqref{eq:B}, we obtain \eqref{eq:CC}.

\end{proof}

\begin{lemma}
\label{lem:C}
Let $A$ be a $d\times d$ expansive integral matrix with $|\det (A)|=2$, and let $f$ be a $2\pi$-periodical continuous function on $\R^d.$
Then
\begin{equation*}
\int_{A \Gamma_\pi} f d \vt = 2 \int_{\Gamma_\pi} f d \vt
\end{equation*}
\end{lemma}
\begin{proof}
By Proposition \ref{prop:see02} Equation \eqref{eq:prop3.1eq3} we have $V_j, U_i \in \mathfrak{G}$ such that
\begin{equation*}
   A = V_1 V_2 \cdots V_{n_0} D_d U_1 U_2 \cdots U_{m_0}.
\end{equation*}
Repeatedly using Lemma \ref{lem:A} Equation \eqref{eq:B} we have
\begin{eqnarray*}
\int_{A \Gamma_\pi} f d \vt  &=& \int_{V_1 V_2 \cdots V_{n_0} D_d U_1 U_2 \cdots U_{m_0} \Gamma_\pi} f d \vt \\
&=& \int_{V_1 V_2 \cdots V_{n_0} D_d \Gamma_\pi} f d \vt  .
\end{eqnarray*}
Denote $\Gamma_\pi ^+ \equiv [-\pi,\pi]^{d-1} \times [0,\pi], \ \Gamma_\pi ^- \equiv [-\pi,\pi]^{d-1} \times [-\pi,0].   $
Modulus a measure zero set we have
$
 D_d \Gamma_\pi = [-\pi,\pi]^{d-1} \times [-2\pi,2\pi]
 =\Gamma_\pi \dotcup \big( (\Gamma_\pi ^+ -2 \pi \vec{e}_d ) \dotcup ( \Gamma_\pi ^- + 2 \pi \vec{e}_d )\big)
$.
So we have
\begin{eqnarray*}
\int_{A \Gamma_\pi} f d \vt  &=&
\int_{V_1 V_2 \cdots V_{n_0} \big( \Gamma_\pi \dotcup \big( (\Gamma_\pi ^+ -2 \pi \vec{e}_d ) \dotcup ( \Gamma_\pi ^- + 2 \pi \vec{e}_d )\big) \big)} f d \vt\\
&=&
\int_{V_1 V_2 \cdots V_{n_0} \Gamma_\pi } f d \vt
+ \int_{V_1 V_2 \cdots V_{n_0} \big( (\Gamma_\pi ^+ -2 \pi \vec{e}_d ) \dotcup ( \Gamma_\pi ^- + 2 \pi \vec{e}_d ) \big)} f d \vt.
\end{eqnarray*}
By Lemma \ref{lem:A}  we have
\begin{eqnarray*}
\int_{V_1 V_2 \cdots V_{n_0} \Gamma_\pi } f d \vt &=& \int_{\Gamma_\pi } f d \vt,
\end{eqnarray*}
and
\begin{eqnarray*}
&& \int_{V_1 V_2 \cdots V_{n_0} \big( (\Gamma_\pi ^+ -2 \pi \vec{e}_d ) \dotcup ( \Gamma_\pi ^- + 2 \pi \vec{e}_d ) \big)} f d \vt \\
&=&
\int_{V_1 V_2 \cdots V_{n_0} ( \Gamma_\pi ^- + 2 \pi \vec{e}_d )} f d \vt +
\int_{V_1 V_2 \cdots V_{n_0}  (\Gamma_\pi ^+ -2 \pi \vec{e}_d )} f d \vt\\
&=&
\int_{V_1 V_2 \cdots V_{n_0} (\Gamma_\pi ^+ \dotcup \Gamma_\pi ^-  )} f d \vt \\
&=& \int_{V_1 V_2 \cdots V_{n_0} \Gamma_\pi} f d \vt \\
&=& \int_{\Gamma_\pi} f d \vt .
\end{eqnarray*}

The proof of Lemma \ref{lem:C} is complete.
\end{proof}

\begin{proposition}
\label{prop:l2phi}
The functions $g$ and $\varphi$ are in $L^2(\R^d).$
\end{proposition}

\begin{proof}
For $J\in\N,$ we define
\begin{equation*}M_J(\vx)\equiv \left\{
\begin{array}{ll}
\prod_{j=1}^{J}|m_0((A^\tau)^{-j}\vx)|^2, & \text{if $\vx\in(A^\tau)^{J+1}\Gamma_\pi$;}\\
0, & \text{if $\vx \in \R^d \backslash (A^\tau)^{J+1}\Gamma_\pi$}.
\end{array}\right.
\end{equation*}
To prove the Proposition \ref{prop:l2phi},  by Fatou's Lemma it suffices to show  that $\{ \int_{\R^d}M_J(\vx)d\vx, J\in\N\}$ is a bounded sequence.

We have
\begin{eqnarray*}
\int_{\R^d}M_J(\vx)d\vx
&=& \int_{(A^\tau)^{J+1}\Gamma_\pi}\prod_{k=1}^{J}|m_0((A^\tau)^{-k}\vx)|^2 d\vx\\
&=& \int_{(A^\tau)^{J}(A^\tau \Gamma_\pi)} \prod_{k=1}^{J}|m_0((A^\tau)^{-k}\vx)|^2 d\vx\\
&=& |\det((A^\tau)^{J})| \int_{A^\tau\Gamma_\pi } \prod_{m=0}^{J-1}|m_0((A^\tau)^{m}\eta)|^2  d\vec{\eta}
\end{eqnarray*}
where  $\vec{\eta}\equiv (A^\tau)^{-J} \vx.$

We claim that
\begin{equation}
  \int_{A^\tau\Gamma_\pi } \prod_{m=0}^{J-1}|m_0((A^\tau)^{m}\eta)|^2  d\vec{\eta}
  =
  \int_{\Gamma_\pi } \prod_{m=1}^{J-1}|m_0((A^\tau)^{m}\eta)|^2  d\vec{\eta}.
\end{equation}
By this claim and the calculation before we will have
\begin{eqnarray*}
  \int_{\R^d}M_J(\vx)d\vx
  &=& |\det((A^\tau)^{J})| \int_{A^\tau\Gamma_\pi } \prod_{m=0}^{J-1}|m_0((A^\tau)^{m}\eta)|^2  d\vec{\eta}\\
  &=& |\det((A^\tau)^{J})| \int_{\Gamma_\pi } \prod_{m=1}^{J-1}|m_0((A^\tau)^{m}\eta)|^2  d\vec{\eta}\\
  &=&  \int_{(A^\tau)^J\Gamma_\pi } \prod_{m=1}^{J-1}|m_0((A^\tau)^{-(J-m)}\vx)|^2  d\vx\\
  &=&   \int_{\R^d}M_{J-1}(\vx)d\vx.
\end{eqnarray*}
So, $\{ \int_{\R^d}M_J(\vx)d\vx, J\in\N\}$ is a constant sequence. We will complete the proof when we finish the proof of the claim.

Proof of the Claim.
By Lemma \ref{lem:sesami}, let $\vec{\gamma}=-\pi \vec{q}_A$, we have
\begin{eqnarray*}
  &&
  \int_{\Gamma_\pi } \prod_{m=0}^{J-1}|m_0((A^\tau)^{m}\eta)|^2  d\vec{\eta}\\
   &=&
  \int_{\Gamma_\pi- \pi \vec{q}_A } \prod_{m=0}^{J-1}|m_0((A^\tau)^{m}\eta)|^2  d\vec{\eta}\\
   &=&
   \int_{\Gamma_\pi} | m_0 (\vec{\lambda}+\pi \vec{q}_A)|^2 \cdot \prod_{m=1}^{J-1}|m_0((A^\tau)^{m}\vec{\lambda} + \pi(A^\tau)^m \vec{q}_A)|^2  d\vec{\lambda}\\
\end{eqnarray*}
where $\vec{\lambda} =  \vec{\eta}-\pi \vec{q}_A.$

%
%
By Partition Theorem equation \eqref{eq:(3)}, $\vec{q}_A \circ A \Z^d $ are even numbers.
So $A^\tau \vec{q}_A \circ \vec{n} $ is even for every $\vec{n} \in\Z^d.$ This implies
that $A^\tau \vec{q}_A \in (2\Z)^d.$ So $(A^\tau)^m\vec{q}_A \in (2\Z)^d$ for $m\geq 1 .$
By the fact that the function $m_0$ is $2\pi$-periodical, we have
\begin{equation}\label{eq:z-2}
 \int_{\Gamma_\pi } \prod_{m=0}^{J-1}|m_0((A^\tau)^{m}\eta)|^2  d\vec{\eta}
   =
   \int_{\Gamma_\pi} |m_0 (\vec{\lambda}+\pi \vec{q}_A)|^2 \cdot \prod_{m=1}^{J-1}|m_0((A^\tau)^{m}\vec{\lambda})|^2  d\vec{\lambda}\\
\end{equation}

Now, by Lemma \ref{lem:C} we have
\begin{equation}
\label{eq:z-1}
    \int_{A^\tau\Gamma_\pi } \prod_{m=0}^{J-1}|m_0((A^\tau)^{m}\eta)|^2  d\vec{\eta}
    =
    2\cdot \int_{\Gamma_\pi } \prod_{m=0}^{J-1}|m_0((A^\tau)^{m}\eta)|^2  d\vec{\eta}.
\end{equation}

Combine \eqref{eq:z-1} and \eqref{eq:z-2}, and by \eqref{eq:m0} we have
\begin{eqnarray*}
&&  \int_{A^\tau\Gamma_\pi } \prod_{m=0}^{J-1}|m_0((A^\tau)^{m}\eta)|^2  d\vec{\eta}\\
&=&
\int_{\Gamma_\pi } \prod_{m=0}^{J-1}|m_0((A^\tau)^{m}\eta)|^2  d\vec{\eta} +
\int_{\Gamma_\pi } \prod_{m=0}^{J-1}|m_0((A^\tau)^{m}\eta)|^2  d\vec{\eta} \\
&=&
\int_{\Gamma_\pi} |m_0 (\vec{\lambda})|^2 \cdot \prod_{m=1}^{J-1}|m_0((A^\tau)^{m}\vec{\lambda})|^2  d\vec{\lambda}
+
\int_{\Gamma_\pi} |m_0 (\vec{\lambda}+\pi \vec{q}_A)|^2 \cdot \prod_{m=1}^{J-1}|m_0((A^\tau)^{m}\vec{\lambda})|^2 d\vec{\lambda}\\
&=&
\int_{\Gamma_\pi} (|m_0 (\vec{\lambda})|^2+ |m_0 (\vec{\lambda}+\pi \vec{q}_A)|^2 )\cdot \prod_{m=1}^{J-1}|m_0((A^\tau)^{m}\vec{\lambda})|^2  d\vec{\lambda}\\
&=&
\int_{\Gamma_\pi} \prod_{m=1}^{J-1}|m_0((A^\tau)^{m}\vec{\lambda})|^2  d\vec{\lambda}.
\end{eqnarray*}
The claim has been proven.

\end{proof}

\section{$\varphi$ has a compact support}

\bigskip

In this section we will prove that the scaling function $\varphi$
has a compact support in $\R^d$ (Proposition \ref{prop:compactsupport}).
We outline the ideas for this. We place the proofs of Lemma \ref{lem:gconverge},
Lemma \ref{lem:ineq} and Proposition \ref{prop:compactsupport} in the Appendix since
in the proofs we use the similar ideas we used in the previous paper \cite{dai2d}

We will need the following Schwartz's Paley-Wiener Theorem.\\

\textsc{Schwartz's Paley-Wiener Theorem}
\textit{An entire function $F$ on $\C^d,d\in\N,$ is the Fourier Transform of a distribution with compact support in $\R^d$ if and only if
there are some constants $C,N$ and $B$, such that
\begin{equation}
\label{eq:ineqcpt}
|F(\vx)| \leq C(1+|\vx|)^N e^{B |\mathfrak{Im} (\vx)|}, ~ \forall \vx \in\C^d
\end{equation}
The distribution is supported on the closed ball with center $\vec{0}$ and radius $B$.
}\\

First, we prove that $g$ is an entire function.
Denote $d_j(\vx) \equiv m_0((A^\tau)^{-j}\vx)-1$ and
denote $\beta=\|(A^\tau)^{-1}\|^{-1}.$ Since $A$ is expansive, $\beta>1$. Here $\|\cdot \|$ is the operator norm of linear operators on the Euclidian space $\C^d.$
\begin{lemma}
\label{lem:gconverge}
Let  $\Omega$ be a bounded closed region in $\C^d.$
Then there exists a constant $C_\Omega>0$,
\begin{equation}
|d_j (\vx)| \leq \frac{C_\Omega}{\beta^j},\ \forall j\in\N,\vx\in \Omega.
\end{equation}
\end{lemma}
(For the proof, see Appendix \textsc{Proof of Lemma \ref{lem:gconverge}}.)\\

\begin{proposition}
\label{prop:entire}
The function $g(\vx)$  is an entire function on $\C^d$.
\end{proposition}
\begin{proof}
For $J\in\N, \ \vx\in\C^d,$ define
\begin{equation}
g_J (\vx) \equiv \frac{1}{(2\pi)^{d/2}}\prod_{j=1}^{J}m_0((A^\tau)^{-j}\vx)= \frac{1}{(2\pi)^{d/2}} \prod_{j=1}^{J}(1+d_j(\vx)).
\end{equation}
Since this is a finite product, $g_J$ is an entire function.

Since $\beta>1,$ by Lemma \ref{lem:gconverge} $\sum |d_j (\vx)|$ converges uniformly on bounded region $\Omega,$
the product $\prod_{j=0}^\infty (1+|d_j(\vx)|)$ converges uniformly on $\Omega.$
This implies that   $g$ is the uniform limit of a sequence of entire functions $g_J$ on every bounded region $\Omega$. By Morera Theorem
$g$ is an entire function on $\C^d.$

\end{proof}

\begin{lemma}
\begin{equation}
\label{eq:E} |e^{-iz}-1|\leq \min(2,|z|), \forall z\in \C,\mathfrak{Im} (z)\leq 0.
\end{equation}
\end{lemma}
\begin{proof}
Let $z=a+ib,$ with $b=\mathfrak{Im} (z)\leq 0.$
So we have
\begin{equation*}
 |e^{-iz}-1| \leq 1+ |e^{-iz}| \leq 1 + e^b \leq 2.
\end{equation*}
On the other hand, we have $e^b > 1+b, \ \forall b\neq 0.$ When $b < 0,$ $b^2=(-b)^2 >(1-e^b)^2.$ So
\begin{equation*}
  |e^{-iz}-1|^2 = e^{2b} -2e^b \cos a +1 =
(e^b-1)^2 +e^b\cdot 4\sin^2 \frac{a}{2} \leq b^2+a^2 = |z|.
\end{equation*}
\end{proof}

\begin{lemma}
\label{lem:ineq}
There exist constants $B_0, C_0$ such that for all $ j\in\N, \ \vx \in \C^d$,
\begin{equation}
   | m_0 \big( (A^\tau)^{-j} \vx \big)| \leq
   \big( 1+  C_0
\min (1, \frac{| \vx|}{\beta^j})\big) \exp( \frac{B_0|\mathfrak{Im}(\vx)|}{\beta^j} )
.
\end{equation}
\end{lemma}
(For the proof, see Appendix \textsc{Proof of Lemma \ref{lem:ineq}}.)

\begin{proposition}
\label{prop:compactsupport}
The scaling function $\varphi$ is an $L^2 (\R^d)$ function with compact support.
\end{proposition}
(For the proof, see Appendix \textsc{Proof of Proposition \ref{prop:compactsupport}}.)\\

\section{Parseval's Frame Wavelet Function $\psi$}\label{ss:psi}

In this section we will define a function $\psi$ associated with the scaling function $\varphi.$  In Theorem \ref{theom_frame} we prove that  the function $\psi$ is a Parseval's frame wavelet function associated with matrix $A$.  Since in the proofs of Lemma \ref{lem:identitylj}, Proposition \ref{prop:go0} and
Proposition \ref{prop:isnorm} we use the ideas that we have used in the previous paper \cite{dai2d}, to make this paper readable, we place these proofs in Appendix.

By Definition  \ref{def:g} and Equations \eqref{eq:R}, \eqref{eq:multiplier} and \eqref{eq:3.6} in Lemma \ref{lem:noproof} we have
\begin{eqnarray*}
  \widehat{\varphi}(\vs)
  &=&g(\vs) =  m_0 ((A^\tau)^{-1} \vs) \cdot \frac{1}{(2\pi)^{d/2}}
  \prod_{j=2}^{\infty}m_0((A^{\tau})^{-j} \vs)\\
  &=& m_0 ((A^\tau)^{-1} \vs) g((A^\tau)^{-1} \vs)\\
  &=&
  \frac{1}{\sqrt{2}}\sum_{\vec{n}\in\Lambda_0} h_{\vec{n}} e^{-i\vec{n}\circ (A^\tau)^{-1}\vs} g((A^\tau)^{-1} \vs)\\
  &=&
  \sum_{\vec{n}\in\Lambda_0} h_{\vec{n}} \widehat{T}_{A^{-1}\vec{n}}\widehat{D}_A g(\vs),\textrm{ by \eqref{eq:multiplier} and   \eqref{eq:3.7} } \\
  &=&
  \sum_{\vec{n}\in\Lambda_0} h_{\vec{n}} \widehat{D}_A\widehat{T}_{\vec{n}} \widehat{\varphi}(\vs), \textrm{by \eqref{eq:R}.}
\end{eqnarray*}
Taking  Fourier inverse transform on two sides, we have
\begin{equation}
\label{eq:2rel}
  \varphi = \sum_{\vec{n}\in\Lambda_0} h_{\vec{n}} D_AT_{\vec{n}} \varphi.
\end{equation}
This is the \textit{two scaling equation associate with matrix $A$}. An equivalent form is
\begin{equation*}
\varphi (\vec{t}) = \sqrt{2} \sum_{\vec{n}\in\Lambda_0} h_{\vec{n}} \varphi (A\vec{t}-\vec{n}), \ \vt\in\R^d.
\end{equation*}

\begin{definition}
\label{def:psi}
Define a function $\psi$ on $\R^d$ by
\begin{equation}
\label{eq:defpsi}
  \psi \equiv \sum_{\vec{n}\in\Z^d}
  (-1)^{\vec{q}_A \circ \vec{n}} \overline{h_{\vec{\ell}_A-\vec{n}}} D_AT_{\vec{n}} \varphi .
\end{equation}  \end{definition}

An equivalent statement is
\begin{equation*}
    \psi (\vec{t}) = \sqrt{2} \sum_{\vec{n}\in\Z^d}
     (-1)^{\vec{q}_A \circ \vec{n}} \overline{h_{\vec{\ell}_A-\vec{n}}}  \varphi (A\vec{t}-\vec{n}),\forall \vec{t}\in\R^d.
\end{equation*}
It is clear that this function $\psi$ has a compact support since $\varphi$ has a compact support and $\{h_{\vec{n}}\}$ is a finite solution to Lawton's equations. For $J\in\Z,$ and $f\in L^2(\R^d)$ define

\begin{eqnarray*}
I_J
 &\equiv&
\sum_{\vec{k}\in\Z^d} \langle f,D_A^JT_{\vec{k}}\varphi\rangle D_A^JT_{\vec{k}}\varphi; \\
F_J
&\equiv&
\sum_{\vec{k}\in\Z^d} \langle f,D_A^JT_{\vec{k}}\psi\rangle D_A^JT_{\vec{k}}\psi.
\end{eqnarray*}

\begin{proposition}
\label{prop:telleskope}
Let $f\in L^2(\R^d)$. Then
\begin{equation}
\label{eq:T}
  I_{J+1} = I_J + F_J, \forall J\in\Z.
\end{equation}
\end{proposition}
(For the proof, see Appendix \textsc{Proof of Proposition \ref{prop:telleskope}}.)\\

For $f\in L^2 (\R^d)$ and $J\in \Z,$ we will use the following notations.
\begin{eqnarray*}
  L_0(f) &\equiv& \sum_{\vec{\ell} \in \Z^d} |\langle f,T_{\vec{\ell}}\varphi \rangle |^2.\\
  L_J(f) &\equiv& \sum_{\vec{\ell} \in \Z^d} |\langle f,D_A^JT_{\vec{\ell}}\varphi \rangle |^2 = L_0 ((D^J _A )^* f).\\
\end{eqnarray*}

Let $\rho>0,$ we define functions $f_\rho$ and
$f_{\overline{\rho}}$ by
\begin{eqnarray*}
  \widehat{f_\rho} &\equiv& \widehat{f} \cdot \chi_{\{|\vt|\leq \rho\}}. \\
  \widehat{f_{\overline{\rho}}}  &\equiv & \widehat{f} \cdot \chi_{\{|\vt| > \rho\}}.
\end{eqnarray*}
Here $\chi$ is the characteristic function.
It is clear that we have
\begin{lemma}
\label{lem:rho}
The function $f_\rho$ and $f_{\overline{\rho}}$ have the following properties,
\begin{enumerate}
  \item $f = f_\rho+f_{\overline{\rho}}$,
  \item $\|f\|^2 = \|\widehat{f}\|^2=\|f_\rho\|^2 +\|f_{\overline{\rho}}\|^2,$
  \item $\lim_{\rho\rightarrow\infty}\|f_\rho\|^2 = \|f\|^2,$
  \item $\lim_{\rho\rightarrow\infty}\|f_{\overline{\rho}}\|^2 = 0.$
\end{enumerate}
\end{lemma}


\begin{theorem}\label{theom_frame}
Let $\psi$ be as defined in Definition \ref{def:psi}. Then,  $\{D_A ^n T_{\vec{\ell}}\psi, n\in\Z,\vec{\ell}\in\Z^d\}$ is a Parseval's frame for $L^2(\R^d)$.
\end{theorem}

\begin{proof}
Let $f\in L^2(\R^d).$ We will prove that
\begin{equation}
\label{eq:psieq2}
f = \sum_{n\in\Z}\sum_{\vec{\ell}\in\Z^d}\langle f, D_A^n T_{\vec{\ell}}\psi\rangle D_A ^n T_{\vec{\ell}}\psi,
\end{equation}
the convergence is in $L^2 (\R^d)$-norm.

By Proposition \ref{prop:telleskope},
we have
$I_j - I_{j-1} =  F_{j-1}$, $\forall j\in\Z$.
Hence
\begin{equation*}
  \sum_{j=-J+1}^{J}F_{j} = I_J - I_{-J}, \forall J\in\Z.
\end{equation*}
This implies that
\begin{eqnarray*}
&&  \sum_{j=-J+1} ^J \sum_{\vec{\ell}\in\Z^d}\langle f, D_A^j T_{\vec{\ell}}\psi\rangle D_A ^j T_{\vec{\ell}}\psi\\
&=&\sum_{\vec{\ell}\in\Z^d}\langle f, D_A^JT_{\vec{\ell}}\varphi\rangle D_A^JT_{\vec{\ell}}\varphi  - \sum_{\vec{\ell}\in\Z^d}\langle f, D_A^{-J}T_{\vec{\ell}}\varphi\rangle D^{-J}T_{\vec{\ell}}\varphi.
\end{eqnarray*}
 Taking inner product of $f$  with both sides of the equation, we have
\begin{eqnarray*}
\sum_{j=-J+1} ^J \sum_{\vec{\ell}\in\Z^d}
|\langle f, D_A^j T_{\vec{\ell}}\psi\rangle|^2
&=& L_J(f) - L_{-J}(f).
\end{eqnarray*}
By Proposition \ref{prop:go0} and Proposition \ref{prop:isnorm}, we have

\begin{align*}
\lim_{J\rightarrow+\infty} & L_J(f)  =  \| f \|^2; \\
\lim_{J\rightarrow+\infty} & L_{-J}(f)  =  0.
\end{align*}
So, we have
\begin{equation*}
\sum_{j\in \Z} \sum_{\vec{\ell}\in\Z^d}|\langle f, D_A^j T_{\vec{\ell}}\psi\rangle|^2
=\|f\|^2, \forall f\in L^2(\R^d).
\end{equation*}

\end{proof}

\begin{proposition}
\label{prop:go0}
Let  $f\in \L^2(\R^d)$. Then
\begin{equation*}
\lim_{J\rightarrow+\infty}
L_{-J}(f)=0.
\end{equation*}
\end{proposition}
(For the proof, see Appendix \textsc{Proof of Proposition \ref{prop:go0}}.)\\

\begin{lemma}
\label{lem:identitylj}
Let $f\in L^2(\R^d)$ and $J\in\Z.$ Then
\begin{equation*}
L_J(f) = (2\pi)^d
\int_{\R^d} \sum_{\vec{\ell}\in\Z^d}  \Big( \widehat{f}(\vt)\overline{\widehat{f}(\vt-2\pi (A^\tau)^J\vec{\ell})}  \widehat{\varphi}((A^\tau)^{-J}\vt-2\pi \vec{\ell}) \overline{\widehat{\varphi}((A^\tau)^{-J}\vt)}\Big) d\vt
\end{equation*}
\end{lemma}
(For the proof, see Appendix \textsc{Proof of Lemma \ref{lem:identitylj}}.)\\

\begin{proposition}
\label{prop:isnorm}
 We have
\begin{equation}
\lim _{J\rightarrow + \infty} L_J(f) = \|f\|^2, \forall f\in L^2(\R^d).
\end{equation}
\end{proposition}
(For the proof, see Appendix \textsc{Proof of Proposition \ref{prop:isnorm}}.)\\

\underline{}\section{Summary and Iteration}

Let $B$ be a $d\times d$ expansive integral matrix with $\det{B}=\pm2$ and $f$ be a function in
$L^2(\R^d).$ Denote
\begin{eqnarray*}
V^{(0)}_B (f) &\equiv& \overline{\textrm{span}}(\{T_{\vec{\ell}}f ,\vec{\ell} \in \Z^d\}).\\
  V^{(n)}_B (f) &\equiv& D_B^nV^{(0)}_B (f), n\in\Z.
\end{eqnarray*}

Let $A_0$ be an $d\times d$ integral expansive matrix with $\det A_0 = \pm 2.$ By the Partition Theorem (Theorem \ref{theom:that}) $A_0 = S^{-1} A S $ for some  $d\times d$ integral matrix $S$ with $\det S = \pm 1$ and there are two vectors $\vec{\ell}_A, \vec{q}_A \in \Z^d$ with properties $(1)\textrm{-}(5)$ in Theorem \ref{theom:that}. By the construction in Sections $\S$\ref{ss:s2}-$\S$\ref{ss:psi} we obtained scaling function $\varphi_A$ and Parseval's frame wavelet in $\psi_A$ $L^2(\R^d)$.

Define
 $\varphi_{A_0}$ and $\psi_{A_0}$ by
 $\varphi_{A_0}(\vt) \equiv (U_S \varphi_A)(\vt) = \varphi_A(S\vt)$
 and
  $\psi_{A_0}(\vt) \equiv (U_S \psi_A)(\vt) = \psi_A(S\vt)$.
 We have
\begin{theorem}
\label{theom:mra}
\begin{enumerate}
  \item The function $\psi_{A_0} $ is a Parseval's frame wavelet associated with matrix $A_0.$
  \item
 \begin{eqnarray}
   \varphi_{A_0} &=& \sum_{\vec{n}\in \Lambda_0} h_{\vec{n}} D_{A_0} T_{S^{-1} \vec{n}} \varphi_{A_0}. \\
    \psi_{A_0} &=& \sum_{\vec{\ell}_A-\vec{n}\in\Lambda_0}
 (-1)^{\vec{q}_A \circ \vec{n}} \overline{h_{\vec{\ell}_A-\vec{n}}} D_{A_0}T_{S^{-1}\vec{n}} \varphi_{A_0} .
\end{eqnarray}
where $\{h_{\vec{n}}, \vec{n} \in \Lambda_0 \}$ is the finite solution to the Lawton's system of equations associated with matrix $A;$ vectors $\{\ell_A, \vec{q}_A\}$ and the $d\times d$ matrix $S$ are as defined in the Partition Theorem (Theorem \ref{theom:that}).
\item
\begin{equation*}
V^{(n)}_{A_0} (\varphi_{A_0})\subset V^{(n+1)}_{A_0} (\varphi_{A_0}), \ n \in \Z.
\end{equation*}
\item
\begin{equation*}
\overline{\bigcup_{n\in\Z} V^{(n)}_{A_0} (\varphi_{A_0})} = L^2(\R^d).
\end{equation*}
\end{enumerate}
\end{theorem}
\begin{proof} By Equations \eqref{eq:2rel} and \eqref{eq:defpsi} we have
\begin{eqnarray*}
 \varphi_A &=& \sum_{\vec{n}\in \Z^d} h_{\vec{n}} D_AT_{\vec{n}} \varphi_A. \\
 \psi_A &=& \sum_{\vec{n}\in \Z^d} (-1)^{\vec{q}_A \circ \vec{n}} \overline{h_{\vec{\ell}_A-\vec{n}}} D_AT_{\vec{n}} \varphi .
\end{eqnarray*}
This implies that for any $\vec{k}\in\Z^d,$
\begin{eqnarray*}
T_{\vec{k}} \varphi_A &=&T_{\vec{k}} D_A\sum_{\vec{n}\in \Z^d } h_{\vec{n}}T_{\vec{n}} \varphi_A
=D_A \sum_{\vec{n}\in \Z^d} h_{\vec{n}}T_{\vec{n}+A\vec{k}} \varphi_A \in V_A ^{(1)}
(\varphi_A)
\end{eqnarray*}
So we have
\begin{equation*}
  V_A ^{(0)} (\varphi_A) \subset V_A ^{(1)} (\varphi_A)
  \textrm{ and } \psi_A \in V_A ^{(1)} (\varphi_A).
\end{equation*}
Hence
\begin{equation*}
    V_A ^{(n)} (\varphi_A) \subset V_A ^{(n+1)} (\varphi_A) , \forall n\in\Z,\textrm{ and }\{T_{\vec{k}} \psi_A, \vec{k}\in\Z^d \} \subset V_A ^{(1)} (\varphi_A).
\end{equation*}
Therefore, we have
\begin{equation*}
  \{D_A ^m T_{\vec{k}} \psi_A, \vec{k}\in\Z^d \} \subset V_A ^{(1)} (\varphi_A), \forall m\in\Z, m\leq 0.
\end{equation*}
Apply $D_A ^n$ to the two sides and take union, we obtain,
\begin{equation*}
\label{eq:mra3}
   \{D_A ^n T_{\vec{k}} \psi_A, n\in\Z, \vec{k}\in\Z^d\} \subset \bigcup_{n\in\Z} V_A ^{(n)} (\varphi_A).
\end{equation*}
This equation together with Theorem \ref{theom_frame} Equation \eqref{eq:psieq2} we get
$\overline{\bigcup_{n\in\Z} V^{(n)}_A (\varphi_A)} = L^2(\R^d).$

Next, notice that $\Lambda_0$ is the finite support of the solution $\{h_{\vec{n}}\}$ to the Lawton's system of equations \eqref{eq:lawtoneq} associated with matrix $A$, by Equations \eqref{eq:2rel}, \eqref{eq:defpsi},  \eqref{eq:P} and \eqref{eq:Q}, we have
\begin{eqnarray*}
   \varphi_{A_0} &=& U_S \varphi_A = U_S \sum_{\vec{n}\in \Z^d} h_{\vec{n}} D_AT_{\vec{n}} \varphi_A \\
  &=& \sum_{\vec{n}\in \Z^d} h_{\vec{n}}  (U_S D_A U_S ^{-1}) (U_S T_{\vec{n}} U_S ^{-1}) (U_S \varphi_A )\\
   &=&\sum_{\vec{n}\in \Z^d} h_{\vec{n}} D_{A_0} T_{S^{-1} \vec{n}} \varphi_{A_0} \\
   &=&\sum_{\vec{n}\in \Lambda_0} h_{\vec{n}} D_{A_0} T_{S^{-1} \vec{n}} \varphi_{A_0}.   \\
    \psi_{A_0} &=&  U_S \psi_A
    =U_S \sum_{\vec{n}\in\Z^d}
 (-1)^{\vec{q}_A \circ \vec{n}} \overline{h_{\vec{\ell}_A-\vec{n}}} D_AT_{\vec{n}} \varphi\\
     &=& \sum_{\vec{n}\in\Z^d} (-1)^{\vec{q}_A \circ \vec{n}} \overline{h_{\vec{\ell}_A-\vec{n}}} (U_S D_A U_S ^{-1})(U_S T_{\vec{n}} U_S ^{-1}) (U_S \varphi_A )\\
    &=& \sum_{\vec{\ell}_A-\vec{n}\in\Lambda_0}
 (-1)^{\vec{q}_A \circ \vec{n}} \overline{h_{\vec{\ell}_A-\vec{n}}} D_{A_0}T_{S^{-1}\vec{n}} \varphi_{A_0} .
\end{eqnarray*}

Since $U_S$ is a unitary operator. It will map a normalize tight frame into a normalized tight frame.
We have
$U_S \{D_A ^n T_{\vec{k}} \psi_A, n\in\Z, \vec{k}\in\Z^d\} = \{D_{A_0} ^n T_{\vec{k}} \psi_{A_0}, n\in\Z, \vec{k}\in\Z^d\}.$ So $\psi_{A_0}$ is a Parseval's frame wavelet. It is also true that $U_S V^{(n)}_{A}(\varphi_A) = V^{(n)}_{A_0}(\varphi_{A_0}), \forall n\in\Z.$ So we have (3) and (4).
\end{proof}

Let $f_0(\vt)$ be a bounded function in $L^2(\R^d)$ which is contiguous at $\vec{0}$ and $f_0(\vec{0})=1.$
Define
\begin{eqnarray*}
\label{def:phi_kj}
    g_0 (\vx) &\equiv& \frac{1}{(2\pi)^{d/2}} \cdot
    f_0 (\vx),\\
    g_k (\vx) &\equiv& \frac{1}{(2\pi)^{d/2}} \cdot
    f_0 ((A^\tau)^{-k}\vx) \cdot
    \prod_{j=1}^{k} m_0((A^\tau)^{-j}\vx), \ \forall k \geq 1,
\end{eqnarray*}
and
\begin{equation*}
    \varphi_k \equiv \F^{-1} g_k, \forall k\geq 0.
\end{equation*}
Since $\lim_k f_0((A^\tau)^{-k}\vx)$ is converging to the constant function $1$ uniformly on any given bounded region of $\R^d$ and
$\prod_{j=1}^{k} m_0((A^\tau)^{-j}\vx)$ is also converging uniformly on any given bounded region of $L^2 (\R^d)$ (see the proof of Proposition \ref{prop:entire}),
we have
$\lim g_k =g$ and
$\lim \varphi_k = \varphi.$  The convergence is in $L^2(\R^d)$-norm.
We have
\begin{eqnarray*}
  \widehat{\varphi}_{k+1}(\vs)
  &=&g_{k+1}(\vs)\\
  &=&  m_0 ((A^\tau)^{-1} \vs) \cdot \frac{1}{(2\pi)^{d/2}} \cdot f_0 ((A^\tau)^{-(k+1)}\vs) \cdot
  \prod_{j=2}^{k+1}m_0((A^{\tau})^{-j} \vs)\\
  &=& m_0 ((A^\tau)^{-1} \vs) \cdot g_k((A^\tau)^{-1} \vs \ )\\
  &=&
  \frac{1}{\sqrt{2}}\sum_{\vec{n}\in\Lambda_0} h_{\vec{n}} e^{-i\vec{n}\circ (A^\tau)^{-1}\vs} g_k((A^\tau)^{-1} \vs \ )\\
  &=&
  \sum_{\vec{n}\in\Lambda_0} h_{\vec{n}} \widehat{T}_{A^{-1}\vec{n}}\widehat{D}_A g_k(\vs),\textrm{ by \eqref{eq:multiplier} and   \eqref{eq:3.7} } \\
  &=&
  \sum_{\vec{n}\in\Lambda_0} h_{\vec{n}} \widehat{D}_A\widehat{T}_{\vec{n}} \widehat{\varphi}_k(\vs), \textrm{by \eqref{eq:R}.}
\end{eqnarray*}
Hence, we have

\begin{equation}
\label{eq:iteratefinal_k}
  \varphi_{k+1} = \sum_{\vec{n}\in\Lambda_0} h_{\vec{n}} D_AT_{\vec{n}} \varphi_k, \ k=1,2,\cdots,
\end{equation}
and
\begin{equation}
\label{eq:iteratefinal_0}
  \varphi_1 = \sum_{\vec{n}\in\Lambda_0} h_{\vec{n}} D_AT_{\vec{n}} \varphi_0.
\end{equation}
Let $\Psi$ be the linear operator
\begin{equation}\label{eq:Psi}
  \Psi \equiv \sum_{\vec{n}\in\Lambda_0} h_{\vec{n}} D_AT_{\vec{n}}.
\end{equation}
The above discussion proves that if $\widehat{\varphi}_0$ is a bounded $L^2(\R^d)$-function which is continuous at $\vec{0}$ with
value $1,$
then the sequence $\{ \Psi ^k \varphi_0 \}$ converges to the scaling function $\varphi$ in $L^2(\R^d)$-norm.

\begin{theorem}
\label{theom:iterate}
We have
\begin{equation}
\label{eq:iteratelimit}
\lim_{k\rightarrow\infty} \Psi^k \chi_{[0,1)^d} = \varphi
\end{equation}
the limit converges in $L^2(\R^d)$-norm.
\end{theorem}

\begin{proof}
Let $\varphi_0 \equiv \chi_{[0,1)^d}.$ All we need to prove is that $\F \varphi_0 = \frac{1}{(2\pi)^{d/2}}\cdot f_0$
for some function $f_0$ and this $f_0$ is bounded in $\R^d$ and continuous at $\vec{0}.$ We have
\begin{eqnarray*}
  \F \varphi_0 (\vs) &=& \frac{1}{(2\pi)^{d/2}}
  \int_{\R^d}e^{-i\vs \circ\vt}\chi_{[0,1)^d}d\vt \\
    &=& \frac{1}{(2\pi)^{d/2}}
  \int_{[0,1)^d}e^{-i\vs \circ\vt}d\vt \\
  &=& \frac{1}{(2\pi)^{d/2}}
  \prod_{j=1} ^d \big(\frac{e^{is_j}-1}{s_j}\big)
\end{eqnarray*}
The function $\frac{e^{is_j}-1}{s_j}$ is continuous at $s_j=0$ if we define its value by $1$. Also, it is bounded when $s_j$ is real numbers. This proves this theorem.

\end{proof}

\section{Examples}\label{ss:examples}

In this section we will use our methods developed in this paper to construction some examples of wavelet functions in $L^2(\R^3)$.

\textsc{Example 1.} In this example we will construct a Haar wavelet in $L^2(\R^3)$ associated with matrix
\begin{equation*}
A_0 \equiv\left[\begin{array}{ccc} 0 & 1 & 0\\  0 & 0 & 1 \\2 & 0 & 0 \end{array}\right].
\end{equation*}
We have $A_0= S^{-1}AS$ where
\begin{equation*}
A \equiv\left[\begin{array}{ccc} 0 & 2 & -1\\  0 & 0 & 1 \\ 1 & 1 & 0 \end{array}\right] \textrm{ and }
S \equiv\left[\begin{array}{ccc} -1 & 0 & 1 \\  1 & 0 & 0 \\0 & 1 & 0 \end{array}\right].
\end{equation*}
Then $\det (A)=2$ with eigenvalues
$\{\sqrt[3]{2}e^{^{\frac{ik\pi}{3}}},k=0,1,2. \}.$
The matrix $A$ is expansive. We have
\begin{equation*}
A \Z^3 = \big\{ \alpha \left[\begin{array}{ccc} 0 \\  0 \\ 1 \end{array}\right]
+  \beta \left[\begin{array}{ccc}1 \\  1 \\ 0 \end{array}\right] + (2\Z)^3 , \alpha,\beta\in\Z \big\}
= A^\tau \Z^3.
\end{equation*}
Let
\begin{equation*}
    \vec{\ell}_A \equiv \left[\begin{array}{ccc} 1 \\  0 \\0\end{array}\right], \
\vec{q}_A \equiv \left[\begin{array}{cc} 1 \\  1 \\ 0 \end{array}\right].
\end{equation*}
The vectors $\vec{\ell}_A , \vec{q}_A$ and matrix $A$ have the properties (1)-(5) in the Partition Theorem.
In this example we assume that the only non zero elements for $h_{\vec{n}}$ are
at
\begin{equation*}
\vec{n}_0 = \left[\begin{array}{cc} 0 \\  0 \\ 0\end{array}\right]
\textrm{ and } \vec{n}_1 =
\left[\begin{array}{cc} 1 \\  0 \\ 0\end{array}\right]\in \vec{\ell}_A + A\Z^3.
\end{equation*}
So the product $h_{\vec{n}_0} \overline{h}_{\vec{n}_1}$ is not in any of the Equations \eqref{eq:lawtoneq}.  The reduced system of Lawton's Equations is
\begin{equation*}
\left\{\begin{array}{l}
h_{\vec{n}_0} ^2 + h_{\vec{n}_1} ^2 =1\\
h_{\vec{n}_0}  + h_{\vec{n}_1} =\sqrt{2}.
\end{array}\right.
\end{equation*}
The system has one solution $h_{\vec{n}_0}= h_{\vec{n}_1} = \frac{\sqrt{2}}{2}.$
The two scaling relation equation \eqref{eq:2rel} is
\begin{equation}
\label{eq:fixed point}
\varphi_A  = \frac{\sqrt{2}}{2}D_A(I+T_{\vec{n}_1})\varphi_A.
\end{equation}

By Theorem \ref{theom:mra}
\begin{equation*}
\label{eq:fixed point}
\varphi_{A_0}  =  U_S \varphi_A
=  \frac{\sqrt{2}}{2}D_{A_0}(I+T_{S^{-1}\vec{n}_1})\varphi_{A_0}
= \frac{\sqrt{2}}{2}D_{A_0}(I+T_{\vec{e}_3})\varphi_{A_0}.
\end{equation*}
Notice that we have
$(I+T_{\vec{e}_3})\chi_{_{[0,1)^3}}=\chi_{_{[0,1)^2 \times [0,2)}}$
and $\frac{\sqrt{2}}{2}D_{A_0}\chi_{_{[0,1)^2 \times [0,2)}} =
\chi_{_{[0,1)^3}}.$  The function $\chi_{_{[0,1)^3}}$ is the scaling function $\varphi_{A_0}.$ Then the related normalized tight frame (orthogonal) wavelet is
\begin{equation*}
\psi_{A_0} = \chi_{_{Q^+}}-\chi_{_{Q^-}},\textrm{ with }
Q^+ \equiv \chi_{_{[0, 0.5)\times [0,1)^2}}\textrm{ and }
Q^- \equiv \chi_{_{[0.5,1)\times [0,1)^2}}
\end{equation*}
This is a Haar wavelet in $L^2(\R^3).$ By this method, we can find examples of Haar wavelets in any dimension.

\textsc{Example 2.} Let
\begin{equation*}
   A\equiv  \left[
      \begin{array}{ccc}
        -2 & 1 & -2 \\
        1  & 0 & 0 \\
        2  & 0 & 2 \\
      \end{array}
    \right],
    \vec{\ell}_A \equiv \left[
                                   \begin{array}{c}
                                     0 \\
                                     0 \\
                                     1 \\
                                   \end{array}
                                 \right]\textrm{ and }
    \vec{q}_A \equiv \left[
                                   \begin{array}{c}
                                     0 \\
                                     0 \\
                                     1 \\
                                   \end{array}
                                 \right].
\end{equation*}
It is clear that $\det(A) =-2.$ Also, we have
\begin{equation*}
    A\Z^3 = \big\{ \alpha \vec{e}_1
+  \beta \vec{e}_2 + (2\Z)^3 , \alpha,\beta\in\Z \big\}
= A^\tau \Z^3.
\end{equation*}
The vectors $\vec{\ell}_A, \vec{q}_A$ and matrix $A$ satisfy the properties (1)-(5) in the Partition Theorem (Theorem \ref{theom:that}). So, if we have a finite solution to the the Equations \eqref{eq:lawtoneq}, we will have a Parseval's frame wavelet associated with matrix $A$. In this example we assume
$    \Lambda_0 \equiv \big\{ \vec{n} = \alpha \vec{e}_1
+  \beta \vec{e}_2 + \gamma \vec{e}_3 , \ \alpha =0, 1,2,3, \ \beta=0,1, \ \gamma = 0,1, \big\}.$
The corresponding reduced Lawton's system of equations is
\begin{equation*}
\label{eq:lawtoneq3d}
\left\{
    \begin{array}{llc}
      \sum_{\vec{n} \in \Lambda_0} h_{\vec{n}} ^2 &= &1 ,\\
      \sum_{k=0} ^3 (h_{k, 0, 0} \cdot  h_{(1+k),0, 0}+ h_{k, 0, 1} \cdot  h_{(1+k),0, 1}+ h_{k, 1, 0} \cdot  h_{(1+k),1, 0}+ h_{k, 1, 1} \cdot  h_{(1+k),1, 1})&= &0 ,\\
      \sum_{k=0} ^2 (h_{k, 0, 0} \cdot  h_{(2+k),0, 0}+ h_{k, 0, 1} \cdot  h_{(2+k),0, 1}+ h_{k, 1, 0} \cdot  h_{(2+k),1, 0}+ h_{k, 1, 1} \cdot  h_{(2+k),1, 1})&= &0 ,\\
      \sum_{k=0} ^1 (h_{k, 0, 0} \cdot  h_{(3+k),0, 0}+ h_{k, 0, 1} \cdot  h_{(3+k),0, 1}+ h_{k, 1, 0} \cdot  h_{(3+k),1, 0}+ h_{k, 1, 1} \cdot  h_{(3+k),1, 1})&= &0 ,\\
      \sum_{k=0} ^0 (h_{k,0,0} \cdot  h_{(3+k),1,0}+h_{k,0,1} \cdot  h_{(3+k),1,1})                                   &= &0 ,\\
      \sum_{k=0} ^1 (h_{k,0,0} \cdot  h_{(2+k),1,0}+h_{k,0,1} \cdot  h_{(2+k),1,1})                                   &= &0 ,\\
      \sum_{k=0} ^2 (h_{k,0,0} \cdot  h_{(1+k),1,0}+h_{k,0,1} \cdot  h_{(1+k),1,1})                                   &= &0 ,\\
      \sum_{k=0} ^3 (h_{k,0,0} \cdot  h_{k,1,0}+h_{k,0,1} \cdot  h_{k,1,1})                                   &= &0 ,\\
      \sum_{k=1} ^3 (h_{k,0,0} \cdot  h_{(k-1),1,0}+h_{k,0,1} \cdot  h_{(k-1),1,1})                                   &= &0 ,\\
      \sum_{k=2} ^3 (h_{k,0,0} \cdot  h_{(k-2),1,0}+h_{k,0,1} \cdot  h_{(k-2),1,1})                                   &= &0 ,\\
      \sum_{k=3} ^3 (h_{k,0,0} \cdot  h_{(k-3),1,0}+h_{k,0,1} \cdot  h_{(k-3),1,1})                                   &= &0 ,\\
      \sum_{\vec{n} \in \Lambda_0} h_{\vec{n}}                                     &= &\sqrt{2}.\\
    \end{array}
    \right.
\end{equation*}

In the Table \ref{tb:2sol} we have two sets of solutions. The solutions satisfies the equations
with errors less than $10^{-10}$.
\begin{table}[ht]
\center
\begin{tabular}{|l|l|l||r|r|}
\hline
$\alpha$ & $\beta$ & $\gamma$ & Solution set $1$ \hspace{.8cm}    & Solution set $2$ \hspace{.8cm} \\
\hline
$0$ & $0$ & $0$ & 0.00000000000000003754 & -0.00000000000000000294  \\
\hline
$1$ & $0$ & $0$ & 0.08378339374280850000 & 0.03292120287539430000  \\
\hline
$2$ & $0$ & $0$ & 0.49453510790101500000 & -0.13290357845020300000  \\
\hline
$3$ & $0$ & $0$ & 0.00000000000000024969 & 0.00000000000000017890  \\
\hline
$0$ & $1$ & $0$ & 0.00000000000000002218 & 0.00000000000000004947  \\
\hline
$1$ & $1$ & $0$ & 0.35330635188230000000 & 0.55716952051625900000  \\
\hline
$2$ & $1$ & $0$ & -0.22451807131547000000 & 0.24991965790058100000  \\
\hline
$3$ & $1$ & $0$ & 0.00000000000000011746 & -0.00000000000000000691  \\
\hline
$0$ & $0$ & $1$ & 0.00000000000000007270 & -0.00000000000000000396  \\
\hline
$1$ & $0$ & $1$ & 0.16226597620431900000 & 0.04430091724524290000  \\
\hline
$2$ & $0$ & $1$ & -0.25534514772672400000 & 0.09876422297993930000  \\
\hline
$3$ & $0$ & $1$ & -0.00000000000000012892 & -0.00000000000000013295  \\
\hline
$0$ & $1$ & $1$ & 0.00000000000000004295 & 0.00000000000000006657  \\
\hline
$1$ & $1$ & $1$ & 0.68425970262500800000 & 0.74976363753740400000  \\
\hline
$2$ & $1$ & $1$ & 0.11592624905984000000 & -0.18572201823152200000  \\
\hline
$3$ & $1$ & $1$ & -0.00000000000000006065 & 0.00000000000000000514  \\
\hline
\end{tabular}
\bigskip
\caption{Two solutions to Equations 
}\label{tb:2sol}
\end{table}


\bibliographystyle{amsplain}

\begin{thebibliography}{99}

\bibitem{baggett}
L. Baggett, H. Medina and K. Merrill,
\textit{Generalized multi-resolution analyses and a construction procedure for all wavelet sets in Rn},
J. Fourier Anal. Appl. \textbf{5}(6) (1999), 563-573.

\bibitem{wangyang}
E. Belogay and Y. Wang,
\textit{Arbitrarily Smooth Orthogonal Nonseparable Wavelets in $\R^2$},
SIAM J. Math. Anal.  \textbf{30} (3) (1999), 678-697.

\bibitem{benedetto}
J. Benedetto and S. Li,
\textit{The theory of multiresolution analysis frames and applications to filter banks},
Appl. Comput. Harmon. Anal.
\textbf{5}(4) (1998), 389–427.


\bibitem{christensen}
O. Christensen,
\textit{An Introduction to Frames and Riesz Bases, Birkh¡§auser}, Appl. Numer. Harmon. Anal. (2013).

\bibitem{dai2d}
X. Dai,
\textit{Equations For Frame Wavelets In $L^2(\R^2)$}, arXiv:1509.05214v1.

\bibitem{dl}
X. Dai and D.R. Larson,
\textit{Wandering vectors for unitary systems and orthogonal wavelets},
 Mem. Amer. Math. Soc. \textbf{640} (1998).

\bibitem{dls}
X. Dai, D.R. Larson and D. M Speegle,
\textit{Wavelet sets in $\R^n$},
J. Fourier Anal. Appl.  \textbf{3}(4) (1997),
451-456.

\bibitem{dau}
I. Daubechies,
\textit{Ten lectures on wavelets},
CBMS 61,  SIAM, 1992.

\bibitem{bdong}
B. Dong and Z. Shen,
\textit{Construction of biorthogonal wavelets from
pseudo-splines,} J. Approx. Theory \textbf{138} (2006) 211 - 231.


\bibitem{guhan}
Q. Gu and D. Han,
\textit{On multiresolution analysis (MRA) wavelets in $\R^n$},
J. Fourier Anal. Appl.
\textbf{6}(4) (2000), 437-447.

\bibitem{halmos}
P. Halmos,
\textit{Finite-Dimensional Vector Spaces},
Undergraduate Texts in Mathematics,
Springer, 1974.

\bibitem{bhan}
B. Han,
\textit{ Compactly supported tight wavelet frames and orthonormal
wavelets of exponential decay with a general dilation matrix,}
J. Comput. Appl. Math. \textbf{155} (2003), 43-67


\bibitem{hanlarson}
D. Han and D.R. Larson,
\textit{Frames, bases, and group representations},
 Mem. Amer. Math. Soc.   \textbf{697} (2000).

\bibitem{hanwang}
D. Han and Y. Wang,
\textit{Lattice tiling and the Weyl—Heisenberg frames},
Geom. Funct. Anal.
\textbf{11}(4) (2001), 742-758.


\bibitem{hormander}
L. H$\ddot{o}$rmander,
\textit{The Analysis of Linear Partial Differential Operators I,  Distribution Theory and Fourier Analysis},  Second Edition  (1989),
Springer-Verlag,  ISBN 3-540-52345-6.

\bibitem{guido}
I.A. Krishtal, B.D. Robinson, G.L. Weiss and E.N. Wilson,
\textit{Some Simple Haar-Type Wavelets in
Higher Dimensions,}
J. Geom. Anal.
\textbf{17} (2007), no. 1, 87-96.



\bibitem{wangyang2}
J. Lagarias and
Y. Wang,
\textit{Haar Type Orthonormal Wavelet Bases in $\R^2$,}
J. Fourier Anal. Appl.
\textbf{2} (1995)
 Issue 1, 1-14.

\bibitem{lawton}
W. Lawton,
\textit{Tight frames of compactly supported affine wavelets},
J. Math. Phys. \textbf{31} (1990),  1898.

\bibitem{lls}
W. Lawton, S. L. Lee and Zuowei Shen,
\textit{An algorithm for
matrix extension and wavelet construction},
Math. Comp. \textbf{65}, (1996) 723-737.

\bibitem{liang}
R. Liang,
\textit{Some properties of wavelets},
PhD thesis (1998),  UNC-Charlotte.

\bibitem{moshenli}
Q. Mo, Y. Shen and S. Li,
\textit{A new proof of some polynomial inequalities related to pseudo-splines,} Appl. Comput. Harmon. Anal. 23 (2007) 415-418.

\bibitem{wangyang3}
Y. Wang,
\textit{Wavelets, tiling, and spectral sets},
Duke Math. J.
Volume 114, Number 1 (2002), 43-57.
\end{thebibliography}

\bigskip

\section{Appendix}

\textsc{Proof of Lemma \ref{lem:LV}}
\begin{proof}
(1) Right multiply some $S_i$ to $B$ to make the elements in the first row  of the product all non negative.
We denote the product by $B_1$, and the product of the $S_i$ we used by $U_1$ . $U_1$ is in $\mathfrak{M}$.
Note that $S_i^{-1} = S_i.$ We have
\begin{equation*}
    B=B_1 \cdot U_1.
\end{equation*}

(2)
Let $b_{1 j}$ be the elements of the first row of $B_1$. Since $B$ is non singular, the row has some positive terms.
Let $b_{1 j_1}$ be the smallest positive element in the row.
So we can further factor $B$ as following
\begin{equation*}
    B = B_2 \cdot U_2
\end{equation*}
where $B_2=B_1 I_{1 j_1}$ and $U_2=I_{1 j_1} U_1.$ Let $c_{1 j}$ be the first row of $B_2.$ Now $c_{1 1}$ is the smallest positive element in the row. $U_2$ is in $\mathfrak{M}.$
If for some $j_2>1$, $c_{1 j_2}>0,$ we factor further
\begin{equation*}
    B= B_3 \cdot U_3
\end{equation*}
where $B_3= B_2 (1-\Delta_{1 j_2})$ and $U_3=(1+\Delta_{1 j_2}) U_2.$ If $c_{1 j}=0, j\geq 2,$ we are done for this step.

(3) If needed, repeat above process (2), until we get factor $B=L_1 V_1,$
where $L_1=(\ell_{ij})$ is an integral matrix with $\ell_{1 1} >0$ and $\ell_{1,j} = 0, j>1.$ And $V_1$ is a finite product of elements from $\mathfrak{G}$ and $V_1\in \mathfrak{M}$.

(4) We can continue the above steps on the $(d-1)\times (d-1)$ matrix $(\ell_{ij})_{i,j\geq 2}.$ By induction we will have
\begin{equation*}
     B=LV,
\end{equation*}
where $V\in \mathfrak{M}$ and $L$ an integral lower triangular matrix with positive diagonals $\{\ell_{11},\cdots, \ell_{dd}\}$.
\end{proof}

\textsc{Proof of Lemma \ref{lem:see0}}
\begin{proof}
Let $B$ be an integral matrix with $\det B= \pm 1.$ By Lemma \ref{lem:LV}, we have $B=LV$
where $L = (\ell_{i j})$ is lower triangular. So $\ell_{i j}=0,$ if $i<j$ and $\ell_{i i} >0, i=1, \cdots d,$
and $\det(L) = \Pi \ell_{jj} =1.$ Therefore, every diagonal element has value $1.$

Let $i,j \leq d, i > j$. Assume $\ell_{i j} > 0.$ By Equation \eqref{eq:magic}, $\Delta_{i j}^n = 0, \forall n \geq 2.$
We have $(I-\Delta_{i j})^{\ell_{i j}}= I+ \sum_{m=1} ^{\ell_{i j}} (-1)^m  C_{\ell_{i j}}    ^m  (\Delta_{i j})^m
     =(I-\ell_{i j} \Delta_{i j}).$  So $(I-\Delta_{i j})^{\ell_{i j}}
     =(I-\ell_{i j} \Delta_{i j}).$ By Equation \eqref{eq:f2}, this equation is also valid for case $\ell_{i j}$. Hence we have
\begin{equation}
    (I\pm\Delta_{i j})^{\ell_{i j}}
     =(I\pm\ell_{i j} \Delta_{i j}).
\end{equation}
The inverse to $(I\pm\Delta_{i j})^{\ell_{i j}}$ are $(I\mp\ell_{i j} \Delta_{i j}).$

Right multiply $(I- \ell_{d j}  \Delta_{d j})$ to $L$ will subtract the $d^{\texttt{th}}$ column from the $j^{\texttt{th}}$ column in matrix $L$. Since $\ell_{d d}=1$, we have factorization This will result a $0$ value at $d j$ position of the product and all other
entries are remain unchanged. So,
\begin{eqnarray*}
  B &=& LV=\big( L  \Pi_{j<d} (I-\Delta_{d j})^{\ell_{d j}}\big) \big( \Pi_{j<d} (I+\Delta_{d j})^{\ell_{d j}} V \big)\\
   &=&
   \big( L  \Pi_{j<d} (I-\Delta_{d j})^{\ell_{d j}}
   \cdot
   \Pi_{j<d-1} (I-\Delta_{(d-1) j})^{\ell_{(d-1) j}} \cdots
   \Pi_{j<2} (I-\Delta_{2 j})^{\ell_{2 j}}
   \big)\cdot\\
   && \big( \Pi_{i=2} ^d \Pi_{j<i} (I+\Delta_{i j})^{\ell_{i j}} V \big)
\end{eqnarray*}
The product in the first big parentheses is $I$ since every $\ell_{i j}, i>j$ is changed into $0$.
By Lemma \ref{lem:LV}, $V$ is the finite product of elements from $\mathfrak{G}$, hence, $B$ can be factored
as products of elements from $\mathfrak{G}$.

\end{proof}

\textsc{Proof of Theorem \ref{theom:redsym}}.
\begin{proof}
We have   $B=S^{-1}AS, D_B=U_B=U_{S^{-1}AS} = U_{S} U_A U_{S^{-1}}= U_SD_A U_S^{-1}$.
Let $f\in L^2 (\R^d)$.
Since $\psi_A$ is a Parseval's frame wavelet associated with the matrix $A$, we have
\begin{eqnarray*}
  U_S ^{-1}  f &=& \sum _{n\in \Z, \vec{\ell}\in\Z^d}
  \langle   U_S ^{-1} f, D_A ^n  T_{\vec{\ell}} \psi_A \rangle D_A ^n  T_{\vec{\ell}} \psi_A.
  \end{eqnarray*}
Then
\begin{eqnarray*}
 f  &=& \sum _{n\in \Z, \vec{\ell}\in\Z^d}
 \langle   f, U_S D_A ^n  T_{\vec{\ell}} \psi_A \rangle
 U_S D_A ^n  T_{\vec{\ell}}\psi_A\\
&=& \sum _{n\in \Z, \vec{\ell}\in\Z^d}
 \langle   f, D_B ^n  T_{S^{-1}\vec{\ell}} \eta_B \rangle
 D_B ^n  T_{S^{-1}\vec{\ell}}\eta_B\\
 &=& \sum _{n\in \Z, \vec{\ell}\in S^{-1}\Z^d}
 \langle   f, D_B ^n  T_{\vec{\ell}} \eta_B \rangle
 D_B ^n  T_{\vec{\ell}}\eta_B.
 \end{eqnarray*}
Since $S$ is an integral matrix with  $|\det (S)|=1,$ we have $\Z^d = S \Z^d = S^{-1} \Z^d.$
So we have
\begin{eqnarray*}
f &=& \sum _{n\in \Z, \vec{\ell}\in\Z^d}
 \langle   f, D_B ^n  T_{\vec{\ell}} \eta_B \rangle
 D_B ^n  T_{\vec{\ell}}\eta_B.
\end{eqnarray*}

\end{proof}

\textsc{Proof of Proposition \ref{prop:filter}}.
\begin{proof} We have
\begin{eqnarray*}
& &|m_0(\vec{t})|^2+|m_0(\vec{t}+\pi\vec{q}_A)|^2
=  \frac{1}{2}\left| \sum_{\vec{m}\in\Z^d} h_{\vec{m}} e^{-i \vec{m}\circ\vec{t}}\right|^2 + \frac{1}{2}\left| \sum_{\vec{m}\in\Z^d} h_{\vec{m}} e^{-i \vec{m}\circ(\vec{t}+\pi\cdot \vec{q}_A)}\right|^2 \\
 &=& \frac{1}{2}\left[ \sum_{\vec{m}\in\Z^d,\vec{n}\in\Z^d} h_{\vec{m}}\overline{h_{\vec{n}}} e^{-i(\vec{m}-\vec{n})\circ\vec{t}} +\sum_{\vec{m}\in\Z^d,\vec{n}\in\Z^d} (-1)^{(\vec{m}-\vec{n})\circ \vec{q}_A} h_{\vec{m}}\overline{h_{\vec{n}}} e^{-i(\vec{m}-\vec{n})\circ\vec{t}}  \right] \\
 &=&
 \frac{1}{2}\left[
 \sum_{\vec{m}\in\Z^d,\vec{k}\in\Z^d}  h_{\vec{m}}\overline{h_{\vec{m}+\vec{k}}} e^{i\vec{k}\circ\vec{t}}+\sum_{\vec{m}\in\Z^d,\vec{k}\in\Z^d} (-1)^{-\vec{k}\circ \vec{q}_A} h_{\vec{m}}\overline{h_{\vec{m}+\vec{k}}} e^{i\vec{k}\circ\vec{t}}
 \right] \\
\end{eqnarray*}
Here $\vec{k}=\vec{m}-\vec{n}.$

By property \eqref{eq:(3)} in Partition Theorem,
$\vec{k}\circ \vec{q}_A$ is odd when $\vec{k}\in(\ell_A+A \Z^d).$ Terms $(-1)^{-\vec{k}\circ \vec{q}_A} h_{\vec{m}}\overline{h_{\vec{m}+\vec{k}}} e^{i\vec{k}\circ\vec{t}}$ in the second sum cancel terms
$ h_{\vec{m}}\overline{h_{\vec{m}+\vec{k}}} e^{i\vec{k}\circ\vec{t}}$ in the first sum. The term  $\vec{k}\circ \vec{q}_A$ is even when $\vec{k}\in A \Z^d.$ So by Lawton's equations \eqref{eq:lawtoneq} we have
\begin{eqnarray*}
&& |m_0(\vec{t})|^2+|m_0(\vec{t}+\pi\vec{q}_A)|^2
 =
 \sum_{\vec{m}\in\Z^d,\vec{k}\in A\Z^d}  h_{\vec{m}}\overline{h_{\vec{m}+\vec{k}}} e^{i\vec{k}\circ\vec{t}} \\
 &=&
 \sum_{\vec{k}\in A\Z^d} \left( \sum_{\vec{m}\in\Z^d}
  h_{\vec{m}}\overline{h_{\vec{m}+\vec{k}}} \right) e^{i\vec{k}\circ\vec{t}}
 =
 \sum_{\vec{k}\in A\Z^d} \delta_{\vec{0} \vec{k}} e^{i\vec{k}\circ\vec{t}}
 = 1.
\end{eqnarray*}
\end{proof}

\textsc{Proof of Lemma \ref{lem:gconverge}}.
\begin{proof}
For $z\in\C,$ define
\begin{equation}
v(z)=\left\{
\begin{array}{ll}
\frac{e^z-1}{z} \quad &z\neq0 \\
 1 \quad &z=0.
\end{array} \right.
\end{equation}
The function $v(z)$ is an entire function on $\C.$

By definition of $m_0$ we have
\begin{align*}
|d_j(\vx)| &= |m_0((A^\tau)^{-j}\vx)-1|\\
           &= \left|\frac{1}{\sqrt{2}}\sum_{\vec{n}\in\Z^d} h_{\vec{n}} e^{-i\vec{n}\circ (A^\tau)^{-j}\vx}-1 \right| \\
           &= \left|\frac{1}{\sqrt{2}}\sum_{\vec{n}\in\Lambda_0} h_{\vec{n}} (e^{-i\vec{n}\circ (A^\tau)^{-j}\vx}-1) \right| \\
           &= \left|\frac{1}{\sqrt{2}}\sum_{\vec{n}\in\Lambda_0} h_{\vec{n}} v(-i\vec{n}\circ(A^\tau)^{-j}\vx)[-i\vec{n}\circ (A^\tau)^{-j}\vx]\right|  \\
           &\leq \frac{1}{\sqrt{2}}\sum_{\vec{n}\in\Lambda_0} |h_n| |v(-i\vec{n}\circ(A^\tau)^{-j}\vx)|
           \cdot |-i\vec{n}\circ (A^\tau)^{-j}\vx|.
\end{align*}
The above estimate on $|d_j(\vx)|$ is bounded by $C_\Omega \cdot \frac{1}{\beta^j}$ for some constant $C_\Omega$
by the following obvious facts combined.
\begin{enumerate}
  \item $|h_n| \leq 1$.
  \item For each $n\in\Lambda_0,$ the function $|v(-i\vec{n}\circ(A^\tau)^{-j}\vx)|$ is continuous and assume its maximum on $\vx\in\Omega.$ Since $\Lambda_0$ is a finite set, $\max\{|v(-i\vec{n}\circ(A^\tau)^{-j}\vx)| , \vx\in\Omega, \vec{n}\in \Lambda_0\}$ is a finite number.
  \item For each $\vec{n}\in\Lambda_0, \vx\in\Omega,$
\begin{equation*}
  |-i\vec{n}\circ (A^\tau)^{-j}\vx|\leq \max \{|\vec{n}|\cdot |\vx|,\ \vx\in\Omega  \text{ \rm and }  \vec{n}\in\Lambda_0\} \cdot \frac{1}{\beta^j},
\end{equation*}
where $\max \{|\vec{n}|\cdot |\vx|,\ \vx\in\Omega  \text{ \rm and }  \vec{n}\in\Lambda_0\}$ is a finite number.
\end{enumerate}

\end{proof}

\textsc{Proof of Lemma \ref{lem:ineq}}.
\begin{proof}
Let $j\in\N, \vx\in \C^d$ and $(A^\tau)^{-j}\vx = \left(\begin{array}{c} \xi_1 \\ \cdots \\ \xi_d  \end{array}\right)\in\C^d.$
Define $ \vec{\ell}_{\vx} = \left(\begin{array}{c} \ell_1 \\ \cdots \\ \ell_d  \end{array}\right)\in\Z^d$ by
\begin{equation*}
\vec{\ell}_m =
\left\{
\begin{array}{rl}
-N_0,        & \text{if } \mathfrak{Im}(\xi_m ) \leq 0; \\
N_0,    & \text{if } \mathfrak{Im} (\xi_m ) > 0.
\end{array}
\right.  m=1,2, \cdots, d.
\end{equation*}
then
$\mathfrak{Im}\Big((\vec{n}-\vec{\ell}_{\vx}) \circ \big((A^\tau)^{-j} \vx\big)\Big)\leq 0$
for $\vec{n}\in\Lambda_0.$
Denote $B_0 \equiv 2\sqrt{d}N_0$.
It is clear that $|\vec{\ell}_\xi| \leq B_0$ and $|(\vec{n} - \vec{\ell}_{\vx})| \leq B_0, \forall n\in\Lambda_0.$
By inequality \eqref{eq:E} we have
\begin{eqnarray*}
 |e^{-i(\vec{n}-\vec{\ell}_{\vx}) \circ \big((A^\tau)^{-j} \vx\big)} -1|
 & \leq &
 \min (2,  |(\vec{n}-\vec{\ell}_{\vx}) \circ \big((A^\tau)^{-j} \vx\big)|),\forall \vec{n}\in\Lambda_0
\end{eqnarray*}
So, we have
\begin{equation}
\label{eq:XX}
|e^{-i(\vec{n}-\vec{\ell}_{\vx}) \circ \big((A^\tau)^{-j} \vx\big)} -1|
\leq \min \big(2, \frac{B_0 | \vx|}{\beta^j}\big) ,\forall \vec{n}\in\Lambda_0.\\
\end{equation}
We also have
\begin{eqnarray*}
 |e^{-i \vec{\ell}_{\vx} \circ \big((A^\tau)^{-j} \vx\big)}|
 &=& e^{\vec{\ell}_{\vx} \circ \big((A^\tau)^{-j} \mathfrak{Im}(\vx) \big)}\\
 &\leq& e^{|\vec{\ell}_{\vx}|  \|(A^\tau)^{-1}\|^j |\mathfrak{Im}(\vx)| }.
 \end{eqnarray*}
This implies
\begin{equation}
\label{eq:YY}
 |e^{-i \vec{\ell}_{\vx} \circ \big((A^\tau)^{-j} \vx\big)}|
 \leq \exp (\frac{B_0  |\mathfrak{Im}(\vx)|}{\beta^j} ).
\end{equation}
Since
\begin{eqnarray*}
m_0 \big( (A^\tau)^{-j} \vx \big)
            & =& e^{-i\vec{\ell}_{\vx} \circ \big((A^\tau)^{-j} \vx\big)}
          \Big( 1+  \sum_{\vec{n}\in\Lambda_0}\frac{1}{\sqrt{2}}h_{\vec{n}}
\big(e^{-i(\vec{n}-\vec{\ell}_{\vx}) \circ \big((A^\tau)^{-j} \vx\big)} -1\big)\Big),
\end{eqnarray*}
by \eqref{eq:XX} and \eqref{eq:YY} we obtain
\begin{eqnarray*}
|m_0 \big( (A^\tau)^{-j} \vx \big)|
&\leq& |e^{ -i\vec{\ell}_{\vx} \circ \big((A^\tau)^{-j} \vx\big)}|\cdot
          \big( 1+  \sum_{\vec{n}\in\Lambda_0}\frac{1}{\sqrt{2}}|h_{\vec{n}}|\cdot
|
e^{-i(\vec{n}-\vec{\ell}_{\vx}) \circ \big((A^\tau)^{-j} \vx\big)} -1 |\big) \\ 
&\leq& \exp (\frac{B_0  |\mathfrak{Im}(\vx)|}{\beta^j} )
\big( 1+  \frac{1}{\sqrt{2}}(2N_0+1)^d
\min (2, \frac{B_0 | \vx|}{\beta^j})\big) \\
&\leq&
\big( 1+  C_0
\min (1, \frac{| \vx|}{\beta^j})\big)
\exp (\frac{B_0  |\mathfrak{Im}(\vx)|}{\beta^j} ).
\end{eqnarray*}
where
$C_0 \equiv \max (\sqrt{2}(2N_0+1)^d, \frac{B_0}{\sqrt{2}}(2N_0+1)^d ).$

\end{proof}

\textsc{Proof of Proposition \ref{prop:compactsupport}}.
\begin{proof}
By Schwartz's Paley-Wiener Theorem, it suffices to prove that the function $g$ satisfies the inequality
\eqref{eq:ineqcpt}.

Let $\vx\in\R^d$. We assume $ \vx\neq \vec{0}.$ (The case $\vx = \vec{0}$ is trivial.)
Denote $B \equiv \sum_{j=1} ^\infty \frac{B_0}{\beta^j}.$ By Lemma \ref{lem:ineq} we have
 \begin{eqnarray*}
  |g(\vx)|
  &=& \big| \frac{1}{(2\pi)^{d/2}}\prod_{j=1}^{\infty}m_0((A^{\tau})^{-j} \vec{\xi})\big|\\
    &\leq& \frac{1}{(2\pi)^{d/2}} e^{B|\mathfrak{Im}(\vx)| }
\prod_{j=1}^{\infty}
  \big( 1+  C_0\min (1, \frac{|\vx|}{\beta^j})\big).
\end{eqnarray*}
On the other hand, the sequence $\{\beta^{j}\}$ is monotonically increasing to $+\infty$.
 Let $I_j \equiv [\beta^j,\beta^{j+1}), j\in\N$ and $I_0 \equiv (0,\beta).$
The set of intervals $\{ I_j, j\geq 0 \}$ is a partition of $(0,\infty).$
So $|\vx| \in I_{j_0}$ for some integer $j_0 \geq 0.$
We have
\begin{eqnarray*}
  (1+C_0)^{j_0} &=& (\beta^{j_0})^{ \log _{\beta} (1+C_0)}\\
&\leq& |\vx|^{\log _{\beta} (1+C_0)}\\
&\leq& (1+|\vx|)^N,
\end{eqnarray*}
where $N $ is the smallest natural number no less than $ \log _{\beta} (1+C_0).$ This is a constant related to $A$ and $N_0$ only.
So, we have
 \begin{eqnarray*}
  |g(\vx)|
  &\leq& \frac{1}{(2\pi)^{d/2}} (1+C_0)^{j_0} e^{B|\mathfrak{Im}(\vx)| }
\prod_{j=j_0+1}^{\infty}
  \big( 1+  C_0\min (1, \frac{| \vx|}{\beta^j})\big)\\
  &\leq& (1+|\vx|)^N e^{B|\mathfrak{Im}(\vx)| }\cdot \Big(\frac{1}{(2\pi)^{d/2}}
\prod_{j=j_0+1}^{\infty}
  \big( 1+  C_0\min (1, \frac{| \vx|}{\beta^j})\big)\Big).
\end{eqnarray*}
Since $|\vx|\in I_{j_0}=[\beta^{j_0},\beta^{j_0+1}), \frac{|\vx|}{\beta^{j_0+1}} < 1$.
We have
\begin{eqnarray*}
&& \frac{1}{(2\pi)^{d/2}}
\prod_{j=j_0+1}^{\infty}
  \big( 1+  C_0\min (1, \frac{| \vx|}{\beta^j})\big)\\
=&&
  \frac{1}{(2\pi)^{d/2}} \prod_{j=j_0+1}^{\infty}
  (1+C_0 \frac{| \vx|}{\beta^{j_0+1}} \cdot \frac{1}{\beta^{j-(j_0+1)}})\\
\leq&&
 \frac{1}{(2\pi)^{d/2}} \prod_{k=0}^{\infty}
  (1+   \frac{C_0}{\beta^{k}})\\
\leq&&
  \frac{1}{(2\pi)^{d/2}}
  e^{\sum \frac{C_0}{\beta^{k}}}.
\end{eqnarray*}
Denote $C \equiv \frac{1}{(2\pi)^{d/2}}
  e^{\sum \frac{C_0}{\beta^{k}}}.$
This is a constant decided by the matrix $A$.\\

Combining the above argument, we have the desired inequality
\begin{equation*}
|g(\vx)| \leq C (1+|\vx|)^N e^{B|\mathfrak{Im}(\vx)| }.
\end{equation*}
\end{proof}

\textsc{Proof of Proposition \ref{prop:telleskope}}.

\begin{proof}
By the equations (\ref{eq:2rel}) and  \eqref{eq:R} $T_{\vec{k}} D_A = D_A T_{A\vec{k}}$, we have
\begin{eqnarray*}
  I_{-1} &=&
  \sum_{\vec{k}\in\Z^d} \langle f, D^{-1} _A T_{\vec{k}}\varphi \rangle D_A ^{-1} T_{\vec{k}}\varphi \\
  &=&
  \sum_{\vec{k}\in\Z^d} \langle f, D_A ^{-1} T_{\vec{k}} \sum_{\vec{p}\in\Z^d} h_{\vec{p}} D_AT_{\vec{p}}\varphi \rangle D_A ^{-1} T_{\vec{k}} \sum_{\vec{q}\in\Z^d} h_{\vec{q}} D_AT_{\vec{q}}\varphi \\
    &=&\sum_{\vec{p}\in\Z^d} \sum_{\vec{q}\in\Z^d} \sum_{\vec{k}\in\Z^d}
 \overline{h_{\vec{p}}} h_{\vec{q}}\langle f,  T_{\vec{p}+A\vec{k}}\varphi \rangle   T_{\vec{q}+A\vec{k}}\varphi \\
     &=&\sum_{\vec{p}\in\Z^d} \sum_{\vec{q}\in\Z^d} \sum_{\vec{k}\in A\Z^d}
 \overline{h_{\vec{p}}} h_{\vec{q}}\langle f,  T_{\vec{p}+\vec{k}}\varphi \rangle   T_{\vec{q}+\vec{k}}\varphi . \\
 F_{-1}
& =& \sum_{\vec{k}\in\Z^d} \langle f, D_A ^{-1} T_{\vec{k}}\psi \rangle D_A ^{-1} T_{\vec{k}}\psi \\
& =& \sum_{\vec{k}\in\Z^d} \langle f, D_A ^{-1} T_{\vec{k}} \sum_{\vec{p}\in\Z^d} (-1)^{\vec{q}_A \circ \vec{p}} \overline{h_{\vec{\ell}_A-\vec{p}}} D_AT_{\vec{p}}\varphi \rangle D_A ^{-1} T_{\vec{k}} \sum_{\vec{q}\in\Z^d} (-1)^{\vec{q}_A \circ \vec{q}} \overline{h_{\vec{\ell}_A-\vec{q}}} D_AT_{\vec{q}}\varphi  \\
& =& \sum_{\vec{p}\in\Z^d}\sum_{\vec{q}\in\Z^d}\sum_{\vec{k}\in\Z^d} (-1)^{\vec{q}_A \circ (\vec{p}+\vec{q})}h_{\vec{\ell}_A-\vec{p}}
\overline{h_{\vec{\ell}_A-\vec{q}}}
\langle f,  T_{\vec{p}+A\vec{k}}\varphi \rangle   T_{\vec{q}+A\vec{k}}\varphi\\
& =& \sum_{\vec{p}\in\Z^d}\sum_{\vec{q}\in\Z^d}\sum_{\vec{k}\in A\Z^d} (-1)^{\vec{q}_A \circ (\vec{p}+\vec{q})}h_{\vec{\ell}_A-\vec{p}}
\overline{h_{\vec{\ell}_A-\vec{q}}}
\langle f,  T_{\vec{p}+\vec{k}}\varphi \rangle   T_{\vec{q}+\vec{k}}\varphi .
\end{eqnarray*}
Use substitutions $\vec{n} \equiv \vec{q}+ \vec{k}$ and $\vec{m} \equiv \vec{p} - \vec{q},$
we have $\vec{q} \equiv \vec{n}- \vec{k}$ and $\vec{p} \equiv \vec{m} + \vec{n} - \vec{k},$
then
\begin{eqnarray*}
I_{-1}
  &=& \sum_{(\vec{m},\vec{n})\in\Z^d\times \Z^d}  \big( \sum_{\vec{k}\in A\Z^d}
 \overline{h_{\vec{m}+\vec{n}-\vec{k}}} h_{\vec{n}-\vec{k}} \big) \langle f, T_{\vec{m}+\vec{n}}\varphi \rangle   T_{\vec{n}}\varphi.\\
F_{-1}
&=&  \sum_{(\vec{m},\vec{n})\in\Z^d\times \Z^d} \big( \sum_{\vec{k}\in A\Z^d}
(-1)^{\vec{q_A}\circ(\vec{m}+2\vec{n}-2\vec{k})}
h_{\vec{\ell}_A-\vec{m}-\vec{n}+\vec{k}} \overline{h_{\vec{\ell}_A-\vec{n}+\vec{k}}} \big)
\langle f,  T_{\vec{m}+\vec{n}}\varphi \rangle   T_{\vec{n}}\varphi\\
&=&  \sum_{(\vec{m},\vec{n})\in\Z^d\times \Z^d} \big( \sum_{\vec{k}\in A\Z^d}
(-1)^{\vec{q_A}\circ\vec{m}}
h_{\vec{\ell}_A-\vec{m}-\vec{n}+\vec{k}} \overline{h_{\vec{\ell}_A-\vec{n}+\vec{k}}} \big)
\langle f,  T_{\vec{m}+\vec{n}}\varphi \rangle   T_{\vec{n}}\varphi.
\end{eqnarray*}
Denote
\begin{eqnarray*}
  \alpha_{\vec{m},\vec{n}} &\equiv& \sum_{\vec{k}\in A\Z^d}
 \overline{h_{\vec{m}+\vec{n}-\vec{k}}} h_{\vec{n}-\vec{k}},\\
  \beta_{\vec{m},\vec{n}} &\equiv& \sum_{\vec{k}\in A\Z^d}
(-1)^{\vec{q}_A\circ \vec{m}}
h_{\vec{\ell}_A-\vec{m}-\vec{n}+\vec{k}} \overline{h_{\vec{\ell}_A-\vec{n}+\vec{k}}}.
\end{eqnarray*}
We have
\begin{equation}
\label{eq:-1}
    I_{-1} + F_{-1} =
    \sum_{(\vec{m},\vec{n})\in \Z^d \times \Z^d}
    \big( \alpha_{\vec{m},\vec{n}} + \beta_{\vec{m},\vec{n}} \big)
    \langle f,  T_{\vec{m}+\vec{n}}\varphi \rangle   T_{\vec{n}}\varphi.
\end{equation}

By Partition Theorem, Equations \eqref{eq:(1)} and \eqref{eq:(3)}, we have
\begin{equation}
\label{eq:sigma}
(-1)^{\vec{q}_A \circ \vec{m}}
    = \left\{
    \begin{array}{rl}
      1, & \vec{m} \in A\Z^d,  \\
      -1, & \vec{m} \in \vec{\ell}_A+A\Z^d.
    \end{array}
    \right.
\end{equation}

For $\vec{m}\in A\Z^d,$ by \eqref{eq:sigma} and Partition Theorem Equation \eqref{eq:partition} we have
\begin{eqnarray*}
   \alpha_{\vec{m},\vec{n}} + \beta_{\vec{m},\vec{n}}  &=&
   \sum_{\vec{k}\in A\Z^d}
 \overline{h_{\vec{m}+\vec{n}-\vec{k}}} h_{\vec{n}-\vec{k}} +
 \sum_{\vec{k}\in A\Z^d}
h_{\vec{\ell}_A-\vec{m}-\vec{n}+\vec{k}} \overline{h_{\vec{\ell}_A-\vec{n}+\vec{k}}}\\
&=&
   \sum_{\vec{\ell}\in \vec{n}-A\Z^d}
 \overline{h_{\vec{m}+\vec{\ell}}} h_{\vec{\ell}} +
 \sum_{\vec{\ell}\in \vec{\ell}_A-\vec{m}-\vec{n}+A\Z^d}
h_{\vec{\ell}} \overline{h_{\vec{m}+\vec{\ell}}}\\
     &=&
    \sum_{\vec{\ell}\in \Z^d}
 \overline{h_{\vec{m}+\vec{\ell}}} h_{\vec{\ell}}.
\end{eqnarray*}
Then by Lawton's equations \eqref{eq:lawtoneq} we obtain
\begin{equation}
\label{eq:first}
    \alpha_{\vec{m},\vec{n}} + \beta_{\vec{m},\vec{n}}
    =
   \delta_{\vec{m}, \vec{0}}, \ \vec{m} \in A\Z^d.
\end{equation}

For $\vec{m}\in \vec{\ell}_A+ A\Z^d,$ by \eqref{eq:sigma} and Partition Theorem Equation \eqref{eq:eq} we have
\begin{eqnarray*}
   \alpha_{\vec{m},\vec{n}} + \beta_{\vec{m},\vec{n}}  &=&
   \sum_{\vec{k}\in A\Z^d}
 \overline{h_{\vec{m}+\vec{n}-\vec{k}}} h_{\vec{n}-\vec{k}} -
 \sum_{\vec{k}\in A\Z^d}
h_{\vec{\ell}_A-\vec{m}-\vec{n}+\vec{k}} \overline{h_{\vec{\ell}_A-\vec{n}+\vec{k}}}\\
&=&
   \sum_{\vec{\ell}\in \vec{n}-A\Z^d}
 \overline{h_{\vec{m}+\vec{\ell}}} h_{\vec{\ell}} -
 \sum_{\vec{\ell}\in \vec{\ell}_A-\vec{m}-\vec{n}+A\Z^d}
h_{\vec{\ell}} \overline{h_{\vec{m}+\vec{\ell}}}\\
&=&
   \sum_{\vec{\ell}\in \vec{n}-A\Z^d}
 \overline{h_{\vec{m}+\vec{\ell}}} h_{\vec{\ell}} -
 \sum_{\vec{\ell}\in \vec{n}-A\Z^d}
h_{\vec{\ell}} \overline{h_{\vec{m}+\vec{\ell}}}\\
     &=&
    0.
\end{eqnarray*}
Or
\begin{equation}
\label{eq:second}
    \alpha_{\vec{m},\vec{n}} + \beta_{\vec{m},\vec{n}}
    =
   0, \vec{m} \in \vec{\ell}_A+A\Z^d.
\end{equation}
Combine \eqref{eq:-1}, \eqref{eq:first} and \eqref{eq:second} we obtain
\begin{equation*}
    I_{-1} + F_{-1} =
    \sum_{\vec{n}\in \Z^d}
    \langle f,  T_{\vec{n}}\varphi \rangle   T_{\vec{n}}\varphi.
\end{equation*}
This is
\begin{equation}
\label{eq:-1B}
    \sum_{\vec{k}\in\Z^d} \langle f, D^{-1} _A T_{\vec{k}}\varphi \rangle D_A ^{-1} T_{\vec{k}}\varphi +
    \sum_{\vec{k}\in\Z^d} \langle f, D^{-1} _A T_{\vec{k}}\psi \rangle D_A ^{-1} T_{\vec{k}}\psi
    =\sum_{\vec{n}\in \Z^d}
    \langle f,  T_{\vec{n}}\varphi \rangle   T_{\vec{n}}\varphi.
\end{equation}

After performing the following operations on equation \eqref{eq:-1B} we will have the desired
Equation \eqref{eq:T} in  this Propositon:

(1) Replace $f$ by $D^{-J-1}f$ in \eqref{eq:-1B} and then.

(2) Apply $D_A ^{J+1}$ to the two sides.

(3) Use the facts that
$\langle D_A ^{-J-1}f, D^{-1} _A T_{\vec{k}}\varphi \rangle=
\langle f, D^{J} _A T_{\vec{k}}\varphi \rangle,$
$\langle D_A ^{-J-1}f, D^{-1} _A T_{\vec{k}}\psi \rangle=
\langle f, D^{J} _A T_{\vec{k}}\psi \rangle$
and $\langle D_A ^{-J-1}f, T_{\vec{k}}\varphi \rangle=
\langle f, D^{J+1} _A T_{\vec{k}}\varphi \rangle$.

\end{proof}

\textsc{Proof of Proposition \ref{prop:go0}}.

\begin{proof}   (1)
By Proposition \ref{prop:compactsupport} the scaling function $\varphi$ has a compact support. Let $B$ be a natural number such that the set $[-B,B)^d$ contains the support of $\varphi.$
We will write
$E_0 \equiv [-\frac{1}{2},\frac{1}{2})^d$, $E_B \equiv [-B-\frac{1}{2},B+\frac{1}{2})^d$
and $\Lambda_B \equiv \Z^d \cap [-B,B]^d$. For $\vec{n}\in\Z^d,$ we have $\vec{n} = (2B+1)\vec{\ell}+\vec{d}, \ \vec{\ell}\in\Z^d, \ \vec{d}\in\Lambda_B.$ Here $\vec{\ell}$ and $\vec{d}\in\Lambda_B$ are uniquely determined  by $\vec{n}.$
We will denote $E_{\vec{\ell}, \vec{d}} \equiv E_B + (2B+1)\vec{\ell}+\vec{d}.$ We have the following partitions.

\begin{eqnarray*}
\Z^d &=& \bigcup _{\vec{d}\in\Lambda_B} \bigcup _{\vec{\ell}\in\Z^d}
\big((2B+1)\vec{\ell}+\vec{d}\big).\\
\R^d &=& \bigcup _{\vec{\ell}\in\Z^d}E_{\vec{\ell},\vec{d}},\text{ \rm for each fixed } \vec{d}\in\Lambda_B.
\end{eqnarray*}
So, for a fixed $\vec{d}\in\Lambda_B$ the functions $\chi_{E_{\vec{\ell},\vec{d}}} f,\vec{\ell} \in \Z^d$ have disjoint supports and sum to $f$, hence are orthogonal.
Then we have
\begin{eqnarray*}
L_0(f)
&=&
\sum_{\vec{d} \in \Lambda_B}\sum_{\vec{\ell} \in \Z^d}
|\langle f,T_{(2B+1)\vec{\ell}+\vec{d}}\varphi \rangle |^2
=
\sum_{\vec{d} \in \Lambda_B}\sum_{\vec{\ell} \in \Z^d}
|\langle \chi_{E_{\vec{\ell},\vec{d}}}f,T_{(2B+1)\vec{\ell}+\vec{d}}\varphi \rangle |^2\\
&\leq&
\sum_{\vec{d} \in \Lambda_B}\sum_{\vec{\ell} \in \Z^d} \|\chi_{E_{\vec{\ell},\vec{d}}}f\|^2 \cdot \|\phi\|^2
=(2B+1)^2 \|f\|^2 \|\phi\|^2.
\end{eqnarray*}
So we have
\begin{eqnarray*}
  L_J(f)
  &=&
    \sum_{\vec{\ell} \in \Z^d} |\langle (D_A ^J )^*f, T_{\vec{\ell}}\varphi \rangle |^2
   \leq
 (2B+1)^2 \|\varphi \|^2 \| (D_A ^J )^*f \|^2.
\end{eqnarray*}
This implies

\begin{equation}
\label{eq:x}
 L_J(f)
  \leq
  (2B+1)^2 \|\varphi \|^2 \| f \|^2, \forall J\in\Z
\end{equation}
and
\begin{equation}
\label{eq:y} \lim_{\rho\rightarrow\infty}\limsup_{J\rightarrow + \infty} L_J(f_{\overline{\rho}}) = 0.
\end{equation}

(2)
We have
\begin{eqnarray*}
L_{-J}(f)&=& \sum_{\vec{\ell} \in \Z^d} |\langle  f,D_A^{-J}T_{\vec{\ell} }\varphi \rangle |^2\\
&=&
\sum_{\vec{d} \in \Lambda_B} \sum_{\vec{\ell} \in \Z^d}
 |\langle  f,D_A^{-J}T_{(2B+1)\vec{\ell}+\vec{d} }\varphi \rangle |^2\\
&=&
\sum_{\vec{d} \in \Lambda_B} \sum_{\vec{\ell} \in \Z^d\backslash \{\vec{0}\}}
 |\langle  f,D_A^{-J}T_{(2B+1)\vec{\ell}+\vec{d} }\varphi \rangle |^2+
\sum_{\vec{d} \in \Lambda_B}
 |\langle  f,D_A^{-J}T_{\vec{d} }\varphi \rangle |^2.
\end{eqnarray*}
It is clear that $E_0 \subset E_B+\vec{d}$ and
$(E_B+(2B+1)\vec{\ell}+\vec{d})\cap E_0 = \emptyset, \forall \vec{\ell} \in \Z^d\backslash \{\vec{0}\}$.
The support of the function $D_A^{-J}T_{\vec{\ell} }\varphi $ is contained in
$A^J (E_B+\vec{\ell}).$ We have
\begin{eqnarray*}
&&\sum_{\vec{\ell} \in \Z^d\backslash \{\vec{0}\}}
 |\langle  f,D_A^{-J}T_{(2B+1)\vec{\ell}+\vec{d} }\varphi \rangle |^2 \\
=&&\sum_{\vec{\ell} \in \Z^d\backslash \{\vec{0}\}}
 |\langle \chi_{A^J(E_B+(2B+1)\vec{\ell}+\vec{d}) } f,D_A^{-J}T_{(2B+1)\vec{\ell}+\vec{d} }\varphi \rangle |^2 \\
 \leq&&
 \|\phi\|^2  \cdot \sum_{\vec{\ell} \in \Z^d\backslash \{\vec{0}\}}
\int_{A^J(E_B+(2B+1)\vec{\ell}+\vec{d}) }  |f |^2 d \mu  \\
\leq&&
\|\varphi\|^2  \cdot \int_{\R^d\backslash A^J E_0}  \left| f \right|^2 d \mu.
\end{eqnarray*}
Since $A$ is expansive, $\lim _{J\rightarrow+\infty}A^J E_0 =\R^d,$ $\lim_{J\rightarrow+\infty} \int_{\R^d\backslash A^J E_0}  \left| f \right|^2 d \mu =0.$
So
\begin{equation*}
\lim_{J\rightarrow+\infty} \sum_{\vec{d} \in \Lambda_B} \sum_{\vec{\ell} \in \Z^d\backslash \{\vec{0}\}}
 |\langle  f,D_A^{-J}T_{(2B+1)\vec{\ell}+\vec{d} }\varphi \rangle |^2
=0.
\end{equation*}

(3) To complete the proof of this Proposition, we need to show that
\begin{equation}
\lim_{J\rightarrow+\infty} \sum_{\vec{d} \in \Lambda_B}
 |\langle  f,D_A^{-J}T_{\vec{d} }\varphi \rangle |^2 =0.
\end{equation}
Let $f_N\equiv \chi_{[-N,N]^2} \cdot f.$ Let $\varepsilon>0,$ and choose $N\in\N$ be large
such that $\|f-f_N\|\leq \frac{\varepsilon}{2\|\varphi\|}.$ Then we have
$|\langle  f,D_A^{-J}T_{\vec{d} }\varphi \rangle|\leq |\langle  f_N,D_A^{-J}T_{\vec{d} }\varphi \rangle|+\frac{\varepsilon}{2}.$ Since
\begin{eqnarray*}
|\langle  f_N,D_A^{-J}T_{\vec{d} }\varphi \rangle | &=& |\langle  D_A^{J}f_N,T_{\vec{d} }\varphi \rangle | \\
&=& |\langle  \chi_{A^{-J} [-N,N]^2} D_A^{J}f_N,T_{\vec{d} }\varphi \rangle | \\
&=& |\langle  D_A^{J}f_N,\chi_{A^{-J} [-N,N]^2} T_{\vec{d} }\varphi \rangle |,
\end{eqnarray*}
we have
\begin{eqnarray*}
|\langle  f_N,D_A^{-J}T_{\vec{d} }\varphi \rangle |
&\leq&
\|D_A^{J}f_N\|  \cdot \sqrt{\int_{A^{-J} [-N,N]^d}\left|T_{\vec{d} }\varphi\right|^2d\mu}.
\end{eqnarray*}
since $\lim_{J\rightarrow \infty} \mu(A^{-J} [-N,N]^d)=0$ and $\|D_A^{J}f_N\| \leq \|f\|,$ so $\lim_{J\rightarrow+\infty}|\langle  f,D_A^{-J}T_{\vec{d} }\varphi \rangle |^2 =0$ for each $\vec{d}\in \Lambda_B.$ Since $\Lambda_B$ is a finite set, we have
\begin{equation*}
    \lim_{J\rightarrow+\infty} \sum_{\vec{d} \in \Lambda_B}
 |\langle  f,D_A^{-J}T_{\vec{d} }\varphi \rangle |^2 =0.
\end{equation*}

\end{proof}

\textsc{Proof of Proposition \ref{prop:isnorm}}
\begin{proof}
We denote
\begin{eqnarray*}
  U_J (f) &\equiv&
  (2\pi)^d
\int_{\R^d} |\widehat{f}(\vs)|^2 |\widehat{\varphi}((A^\tau)^{-J}\vs) |^2 d\vs,\\
  V_J (f) &\equiv&
  (2\pi)^d
\int_{\R^d}\sum_{\vec{\ell} \in \Z^d\backslash\{\vec{0}\}}\Big(\widehat{f}(\vs) \cdot
\overline{\widehat{f}(\vs-2\pi(A^\tau)^{J}\vec{\ell})} \cdot
\widehat{\varphi}((A^\tau)^{-J}\vs-2\pi\vec{\ell})\cdot
\overline{\widehat{\varphi}((A^\tau)^{-J}\vs)}
\Big) d\vs.
\end{eqnarray*}
By Lemma \ref{lem:identitylj}, we have
$L_J(f) = U_J (f) + V_J (f).$ It is enough to prove that
\begin{equation}
\label{eq:r}
\lim_{J\rightarrow +\infty} U_J (f) = \| f \| ^2
\end{equation}
and
\begin{equation}
\label{eq:s}
\lim_{J\rightarrow +\infty} V_J (f) = 0.
\end{equation}

1. Recall that $A^\tau$ is also expansive, so $\lim_{J\rightarrow +\infty}(A^\tau)^{-J}\vs=\vec{0}, \forall \vs\in\R^d.$
Also, by definition \eqref{def:phi}, Corollary \ref{less1} and Proposition \ref{prop:entire}, $g$ is bounded and continuous on $\R^d$. Also, $(2\pi)^{\frac{d}{2}} \widehat{\varphi} (\vec{0}) = (2\pi)^{\frac{d}{2}} g(\vec{0})=1.$ By Lebesgue Dominate Convergence Theorem we have
\begin{eqnarray*}
  \lim_{J\rightarrow +\infty} U_J (f) &=&
  \lim_{J\rightarrow +\infty}
(2\pi)^d  \int_{\R^d}|\widehat{f}(\vs)|^2 \cdot
|\widehat{\varphi}((A^\tau)^{-J}\vs)|^2
d\vs\\
&=&   \lim_{J\rightarrow +\infty}
  \int_{\R^d}|\widehat{f}(\vs)|^2 \cdot
|(2\pi)^{\frac{d}{2}}g((A^\tau)^{-J}\vs)|^2
d\vs\\
&=&  \|\widehat{f}\|^2
=  \|f\|^2.
\end{eqnarray*}
This proves \eqref{eq:r}.

2. Let $\rho\in\R^+,$ $\Delta_\rho$ be the open ball with center $\vec{0}$ and radius $\rho.$
Since $A^\tau$ is expansive, $\beta\equiv \|(A^\tau)^{-1}\|^{-1}>1.$ Denote $a\equiv \log_\beta(2\rho).$ Let  $J_{\rho}$ be the smallest natural number in the interval
$(a,+\infty).$
 When $J\geq J_\rho,$  $(A^\tau)^J \Delta_1$ contains an open ball $\Delta_{2\rho}.$
 Since $\Delta_1\cap \Z^d =\{\vec{0}\},$
$2\pi (A^\tau)^J\Delta_1\cap 2\pi (A^\tau)^J\Z^d =\{\vec{0}\}.$ Also we have $\Delta_{2\rho} \subseteq (A^\tau)^J\Delta_1\subseteq 2\pi (A^\tau)^J\Delta_1.$
These facts implies that when $J\geq J_\rho$ the distance between $\vec{0}$ and $2\pi (A^\tau)^J \vec{\ell}$ is greater than $2\rho.$
So for each $\vec{\ell}\in\Z^d\backslash\{\vec{0}\}$, the support of $\widehat{f_\rho}(\vt)$ which is $\Delta_\rho$ and the  support of $\widehat{f_\rho}(\vt-2\pi (A^\tau)^J\vec{\ell})$
which is $\Delta_\rho+2\pi (A^\tau)^J\vec{\ell}$ are disjoint.
This implies that the product
$\widehat{f_\rho}(\vt)\overline{\widehat{f_\rho}(\vt-2\pi (A^\tau)^J\vec{\ell})}\equiv 0$
when $J\geq J_\rho.$ Therefore, we have
\begin{equation*}
    \lim_{J\rightarrow+\infty} V_J (f_\rho) = 0,  \forall \rho\in\R^+.
\end{equation*}
Together with \eqref{eq:r}, we have proved that
\begin{equation}
\label{eq:i}
\lim _{J\rightarrow + \infty} L_J(f_\rho) = \|f_\rho\|^2,
\forall f\in L^2(\R^d),\forall\rho\in\R^+.
\end{equation}

3. Let $D_\rho \equiv \sum_{\vec{\ell}\in\Z^d} \big(\langle f_\rho,D_A ^J T_{\vec{\ell}}\varphi \rangle
  \overline{\langle f_{\overline{\rho}} ,D_A ^J T_{\vec{\ell}}\varphi \rangle}+
  \langle f_{\overline{\rho}} ,D_A ^J T_{\vec{\ell}}\varphi \rangle
  \overline{\langle f_\rho ,D_A ^J T_{\vec{\ell}}\varphi \rangle}\big).$
  Then
  \begin{eqnarray*}
   | D_\rho | &\leq& 2\sum_{\vec{\ell}\in\Z^d} | \langle f_\rho,D_A ^J T_{\vec{\ell}}\varphi \rangle| \cdot
  |\langle f_{\overline{\rho}} ,D_A ^J T_{\vec{\ell}}\varphi \rangle| \\
  &\leq&
  2\sqrt{\sum_{\vec{\ell}\in\Z^d} | \langle f_\rho,D_A ^J T_{\vec{\ell}}\varphi \rangle|^2} \cdot
  \sqrt{\sum_{\vec{\ell}\in\Z^d}|\langle f_{\overline{\rho}} ,D_A ^J T_{\vec{\ell}}\varphi \rangle|^2}\\
  &=&
  2 \sqrt{L_J(f_\rho)} \cdot \sqrt{L_J(f_{\overline{\rho}})}.
  \end{eqnarray*}
By \eqref{eq:x}, we have

\begin{equation}
\label{eq:haha}
 | D_\rho | \leq  2 (2B+1)^2 \|\varphi\|^2 \|f_\rho\| \|f_{\overline{\rho}}\|.
\end{equation}

We have
\begin{eqnarray*}
  L_J(f)-\|f\|^2
  &=& L_J(f_\rho + f_{\overline{\rho}}) -\|f\|^2 \\
  &=& \sum_{\vec{\ell}\in\Z^d} \langle f_\rho + f_{\overline{\rho}} ,D_A ^J T_{\vec{\ell}}\varphi \rangle
  \overline{\langle f_\rho + f_{\overline{\rho}} ,D_A ^J T_{\vec{\ell}}\varphi \rangle}-\|f\|^2\\
  &=& L_J(f_\rho )-\|f\|^2 + L_J(f_{\overline{\rho}})+D_\rho\\
  &=& \big(L_J(f_\rho )-\|f_\rho\|^2\big) -\|f_{\overline{\rho}}\|^2 + L_J(f_{\overline{\rho}})+D_\rho.
\end{eqnarray*}
By \eqref{eq:i}, \eqref{eq:y} and \eqref{eq:haha} we have       
  \begin{eqnarray*}
&&\limsup_{J\rightarrow + \infty}\big| L_J(f)-\|f\|^2\big|\leq 0+ \|f_{\overline{\rho}}\|^2 + 0 +
2 (2B+1)^2 \|\varphi\|^2 \|f_\rho\| \|f_{\overline{\rho}}\|, \forall \rho\in\R^+.
\end{eqnarray*}
Let $\rho\rightarrow+\infty,$
we have \eqref{eq:s}
\begin{eqnarray*}
  \lim_{J\rightarrow \infty} L_J (f)  &=& \|f\|^2.
\end{eqnarray*}

\end{proof}

\end{document}